\documentclass
[production,12pt]{jsg}
\usepackage{latexsym,amsmath,amsfonts,amscd,amssymb,enumerate,url}
\newtheorem{thm}{Theorem}
\newtheorem{prop}[thm]{Proposition}
\newtheorem{lemma}[thm]{Lemma}
\newtheorem{claim}{Claim}

\newtheorem{rmk}[thm]{Remark}
\newtheorem*{Rmk}{Remark}
\newtheorem*{Rmks}{Remarks}

\newtheorem{pb}{Problem}

\newenvironment{pf}[1][Proof.]{\noindent \textbf{#1:} }{}
\newenvironment{enui}{\begin{enumerate}[(i)]}{\end{enumerate}}

\newcommand\Om{\Omega}
\newcommand\om{\omega}
\newcommand\sub{\subseteq}

\newcommand{\N}{\mathbb{N}}
\newcommand\NN{\nu}
\newcommand{\Z}{\mathbb{Z}}
\newcommand{\R}{\mathbb{R}}

\newcommand\F{{\mathcal{F}}}
\newcommand\x{\times}
\newcommand\disj{\bigsqcup}

\newcommand\hol{{\operatorname{hol}}}

\newcommand{\BAR}[1]{{\overline{#1}}}
\newcommand\wt[1]{{\widetilde{#1}}}

\newcommand\OO{\operatorname{O}}
\newcommand\UU{\mathcal{U}}
\newcommand\VV{\mathcal{V}}

\newcommand\wo{\setminus}
\newcommand\Int[1]{\stackrel{\circ}{#1}}
\newcommand\pr{\operatorname{pr}}
\newcommand{\hhat}[1]{\widehat{#1}}

\newcommand\nn{{\nonumber}}

\newcommand\eps{\varepsilon}

\newcommand\id{{\operatorname{id}}}

\renewcommand\phi{\varphi}
\newcommand\codim{{\operatorname{codim}}}
\newcommand\Wlog{w.l.o.g.~}
\newcommand\wrt{w.r.t.~}
\newcommand\Fix{{\operatorname{Fix}}}

\newcommand\dd{\partial}
\newcommand\cont{\supseteq}
\newcommand\dom{\operatorname{dom}}
\newcommand\Ham{\operatorname{Ham}}
\newcommand\cl{\operatorname{cl}}
\newcommand\pt{\operatorname{pt}}
\newcommand\then{\Longrightarrow}
\newcommand\follows{\Longleftarrow}
\newcommand\const{\equiv}

\title[Leafwise fixed points]{Leafwise fixed points for $C^0$-small Hamiltonian flows}

\author[Fabian Ziltener]{Fabian Ziltener\\
Utrecht University, Mathematics Department, Budapestlaan 6, 3584CD Utrecht, The Netherlands\\
e-mail: f.ziltener@uu.nl
}

\begin{document}

\maketitle

\begin{abstract}
Consider a closed coisotropic submanifold $N$ of a symplectic manifold $(M,\om)$ and a Hamiltonian diffeomorphism $\phi$ on $M$. The main result of this article states that $\phi$ has at least the cup-length of $N$ many leafwise fixed points \wrt $N$, provided that it is the time-1-map of a global Hamiltonian flow whose restriction to $N$ stays $C^0$-close to the inclusion $N\to M$. If $(\phi,N)$ is suitably nondegenerate then the number of these points is bounded below by the sum of the Betti-numbers of $N$. The nondegeneracy condition is generically satisfied.

This appears to be the first leafwise fixed point result in which neither $\phi\big|_N$ is assumed to be $C^1$-close to the inclusion $N\to M$, nor $N$ to be of contact type or regular (i.e., ``fibering''). It is optimal in the sense that the $C^0$-condition on $\phi$ cannot be replaced by the assumption that $\phi$ is Hofer-small. 
\end{abstract}

\tableofcontents

\section{Introduction and main result}\label{sec:}
Consider a symplectic manifold $(M,\om)$ and a coisotropic submanifold $N\sub M$. This means that for every $x\in N$ the symplectic complement of $T_xN$,
\[T_xN^\om:=\big\{v\in T_xM,\big|\,\om(v,w)=0,\,\forall w\in T_xN\big\},\]
is contained in $T_xN$. It follows that $TN^\om=\disj_{x\in N}T_xN^\om$ is an involutive distribution on $N$. By Frobenius' Theorem such a distribution gives rise to a foliation on $N$. The leaves of this foliation are called isotropic leaves.

Let $S\sub M$ be a set containing $N$ and $\phi:S\to M$. A \emph{leafwise fixed point for $\phi$} is a point $x\in N$ for which $\phi(x)$ lies in the isotropic leaf through $x$. We denote by $\Fix(\phi,N)$ the set of such points.

For every function $F\in C^\infty(M,\R)$ we denote by $X_F$ its Hamiltonian vector field. It is the unique vector field on $M$ satisfying
\[dF=\om(X_F,\cdot).\]
For every function $H\in C^\infty\big([0,1]\x M,\R\big)$ we denote $H_t:=H(t,\cdot)$ and by $(\phi_H^t)_{t\in[0,1]}$ the Hamiltonian flow of $H$, i.e., the flow of the time-dependent vector field $(X_{H_t})_{t\in[0,1]}$. We denote by $\dom(\phi_H^t)\sub M$ the domain of $\phi_H^t$. By a global Hamiltonian flow on $M$ we mean a family $(\phi_H^t)_{t\in[0,1]}$ arising this way, such that $\dom(\phi_H^t)=M$ and $\phi_H^t$ is surjective, for every $t\in[0,1]$.%
\footnote{In this case $\phi_H^t$ is a diffeomorphism of $M$ for every $t$. As an example this happens if $H$ has compact support.}%
 By a Hamiltonian diffeomorphism on $M$ we mean the time-1-map of a global Hamiltonian flow on $M$. We denote by $\Ham(M,\om)$ the set of Hamiltonian diffeomorphisms on $M$.

Let $\phi\in\Ham(M,\om)$. A fundamental problem in symplectic geometry is the following:
\begin{pb}\label{pb:leafwise} Find conditions under which $\Fix(\phi,N)$ is nonempty and find lower bounds on its cardinality.
\end{pb}
In the extreme case $N=M$ the set $\Fix(\phi,N)$ consists of the usual fixed points of $\phi$. Such points correspond to periodic orbits of Hamiltonian systems. Starting with a famous conjecture by V.~I.~Arnold \cite{Ar}, the above problem has been extensively studied in this case. It has led to Hamiltonian Floer homology.

On the opposite extreme, consider the case in which $N$ is Lagrangian, i.e., has half the dimension of $M$. Then 
\[\Fix(\phi,N)=N\cap\phi^{-1}(N)\]
(provided that $N$ is connected). In this situation, based on seminal work by A.~Floer \cite{FlLag}, the above problem has given rise to Lagrangian Floer homology and the Fukaya category. 

As an intermediate case, coisotropic submanifolds of codimension 1 arise in classical mechanics as energy level sets for an autonomous Hamiltonian. If $\phi$ is the time-one flow of a time-dependent perturbation of the Hamiltonian, then $\Fix(\phi,N)$ corresponds to the set of points on the level set whose trajectory is changed only by a phase shift, under the perturbation.

In order to solve Problem \ref{pb:leafwise}, it is necessary to make some assumptions on $N$ and $\phi$. Firstly, one should assume that $N$ is closed,%
\footnote{This means that $N$ is compact and has no boundary.}%
since e.g.~the interval $(0,1)\x\{0\}$ is a Lagrangian submanifold of $\R^2$ that can be displaced by a Hamiltonian diffeomorphism that is $C^\infty$-close to the identity. In general one also needs that $\phi$ is close to the identity in a suitable sense, since every compact subset of the product of $\R^2$ with a symplectic manifold can be displaced by some Hamiltonian diffeomorphism.

So far solutions to Problem \ref{pb:leafwise} have been found under restrictive assumptions on $\phi$ or $N$, see \cite{Mo,Ban,EH,HoTop,Gi,Dr,Gu,ZiLeafwise,ZiMaslov,AFRab,AFinf,AFSurvey,AMo,AMc,Bae,KaEx,KaGen,MMP,SaIt,SaMorse}. In \cite{Mo,Ban} J.~Moser and A.~Banyaga assumed that the restriction $\phi|_N$ is $C^1$-close to the inclusion $N\to M$. In most other results this condition was relaxed to the condition that $\phi$ be Hofer close to the identity. On the other hand, strong conditions on $N$ were imposed, namely that it is of contact type or regular (i.e., ``fibering''). For a further discussion of the history of the problem see \cite{ZiLeafwise}.

The main result of this article is that the conditions on $N$ (except for closedness) can be removed altogether, if $\phi$ is the time-1-map of a Hamiltonian flow whose restriction to $N$ stays $C^0$-close to the inclusion $N\to M$. To state this result, we need the following. Let $X$ be a topological space. We define its \emph{cup-length $\cl(X)$} to be the infimum of all integers $m\in\{0,1,2,\ldots\}$ with the following property. If $R$ is a commutative ring with unity, 
$k_1,\ldots,k_m\in\{1,2,\ldots\}$, and $a_i\in H^{k_i}(X,R)$,%
\footnote{This denotes the degree $k_i$ singular cohomology of $X$ with coefficients in $R$.}%
 for $i=1,\ldots,m$, then
\[a_1\smile\cdots\smile a_m=0.%
\footnote{Here the cup product of an empty collection of classes is defined to be $1\in H^*(X,R)$. It follows that $\cl(\{\pt\})=1$. While defined slightly differently, $\cl(X)$ equals the usual cup-length and the cohomological category defined in \cite[Definition 2]{HoLag}.}%
\]

Some part of the main result concerns the case in which the pair $(\phi,N)$ is $\OO$-nondegenerate. This notion will be defined in Section \ref{sec:nondeg}. It naturally generalizes the usual nondegeneracy in the case $N=M$ and transversality of the intersection $N\cap\phi^{-1}(N)$ in the Lagrangian case. 
We denote by $b_i(X,\Z_2)$ the $i$-th Betti number of $X$ with $\Z_2$-coefficients.
%
%
\begin{thm}[leafwise fixed points]\label{thm:leafwise} Let $(M,\om)$ be a symplectic manifold and $N\sub M$ be a closed coisotropic submanifold. Then there exists a $C^0$-neighbourhood%
\footnote{This means a neighbourhood in the compact open topology. This topology coincides with the uniform topology induced by the distance function on $M$ coming from any Riemannian metric on $M$.} %
$\UU\sub C(N,M)$ of the inclusion $N\to M$ and a neighbourhood $\OO$ of the diagonal
\[\wt N:=\{(y,y)\,\big|\,y\in N\big\}\]
in $N\x N$ with the following properties. 
\begin{enui}
\item\label{thm:leafwise:Fix} If $(\phi^t)_{t\in[0,1]}$ is a global Hamiltonian flow on $M$ satisfying $\phi^t|_N\in\UU$, for every $t\in[0,1]$, then 
\begin{equation}\label{eq:cl N}\big|\Fix(\phi^1,N)\big|\geq\cl(N).\end{equation}
If in addition the pair $(\phi^1,N)$ is $\OO$-nondegenerate then 
\begin{equation}\label{eq:Fix phi 1 N}\big|\Fix(\phi^1,N)\big|\geq\sum_{i=0}^{\dim N}b_i(N,\Z_2).\end{equation}
\item\label{thm:leafwise:dense} We denote
\begin{equation}\label{eq:Ham M om U}\Ham(M,\om,\UU):=\big\{\phi^1\,\big|\,(\phi^t)\textrm{ global Hamiltonian flow: }\phi^t|_N\in\UU\big\}.\end{equation}
The set
\[\big\{\phi\in\Ham(M,\om,\UU)\,\big|\,(\phi,N)\textrm{ is }\OO\textrm{-nondegenerate}\big\}\]
is dense in $\Ham(M,\om,\UU)$ in the strong (Whitney) $C^\infty$-topology. 
%
\end{enui}
\end{thm}
\begin{Rmks} \begin{itemize}

\item If $N\neq\emptyset$ then $\cl(N)\geq1$. Hence by Theorem \ref{thm:leafwise} the map $\phi^1$ has a leafwise fixed point.


\item Condition (\ref{thm:leafwise:dense}) means that among Hamiltonian diffeomorphisms that arise from a flow whose restriction to $N$ stays $C^0$-close to the inclusion $N\to M$, the $\OO$-nondegenerate ones are ``generic''.

\item The hypothesis that $(\phi^1,N)$ be $\OO$-nondegenerate, is satisfied if $(\phi^1,N)$ is nondegenerate in the sense of \cite{ZiLeafwise} (see Section \ref{sec:nondeg}).

\end{itemize}

\end{Rmks}
This theorem is optimal in the sense that the $C^0$-hypothesis $\phi^t|_N\in\UU$ cannot be replaced by the ``$C^{-1}$-assumption'' that $\phi^t$ is Hofer-small for every $t$: To explain this comment, recall that for $H\in C^\infty(M,\R)$ we have $dH=\om(X_H,\cdot)$. Hence the oscillation $\sup H-\inf H$ can be thought of as a $C^{-1}$-norm for the Hamiltonian vector field $X_H$. Thus the Hofer norm of a Hamiltonian diffeomorphism can roughly be viewed as its ``$C^{-1}$-distance'' to the identity. Optimality now follows from the next remark.
%
%
\begin{Rmk}[Hofer-smallness does not suffice]\label{rmk:Hofer not suffice} V.~L.~Ginzburg and B.~Z.~G\"urel \cite[Theorem 1.1]{GGFrag} recently proved that for $n\geq2$ there exists a closed smooth hypersurface $N\sub\R^{2n}$ and a compact subset $K\sub\R^{2n}$, such that for every $\eps>0$ there exists a smooth function $H:\R^{2n}\to\R$ with support in $K$, such that
\[\max H-\min H<\eps,\quad\Fix\big(\phi_H^1,N\big)=\emptyset.\]
Note that $\phi_H^t$ has Hofer norm bounded above by $\eps t\leq\eps$, for every $t\in[0,1]$.

This theorem and Theorem \ref{thm:leafwise} do not contradict each other, since Hofer-smallness does not imply the hypothesis $\phi^t|_N\in\UU$ of Theorem \ref{thm:leafwise}. More precisely, let $(M,\om)$ be a symplectic manifold of positive dimension, and $\emptyset\neq N\sub M$ a coisotropic submanifold. Then there exists a $C^0$-neighbourhood $\UU$ of the inclusion $N\to M$, such that for every Hofer neighbourhood $\VV\sub\Ham(M,\om)$ of the identity there exists $\phi\in\VV$ for which $\phi|_N\not\in\UU$. 

To show this, we choose $x\in N$ and a point $y\in M\wo\{x\}$ that lies in a connected Darboux chart around $x$. 
We define 
\[\UU:=\big\{\psi\in C(N,M)\,\big|\,\psi(x)\neq y\big\}.\]
This is a $C^0$-neighbourhood of the inclusion $N\to M$. Let $\VV\sub\Ham(M,\om)$ be a Hofer neighbourhood of the identity. There exists $\phi\in\VV$, such that $\phi(x)=y$. We may construct such a Hamiltonian diffeomorphism in any connected Darboux chart containing $x$ and $y$. We have $\phi|_N\not\in\UU$. Hence $\UU$ has the desired property.
\end{Rmk}
\begin{Rmk}[extreme cases] In the extreme cases $N=M$ and $N$ Lagrangian A.~Weinstein \cite{WePert,WeCrit} proved lower bounds on $\big|\Fix(\phi,N)\big|$ for the time-one flow $\phi$ of a $C^0$-small Hamiltonian vector field. Such a $\phi$ satisfies a stronger condition than the one in Theorem \ref{thm:leafwise}(\ref{thm:leafwise:Fix}). 
\end{Rmk}
\begin{Rmk}[translated points] In \cite[p.~96]{SaMorse} S.~Sandon also proved a leafwise fixed point result in a special $C^0$-context (amongst other results). More precisely, she showed a lower bound on the number of translated points of the time-1-map of a $C^0$-small contact isotopy on a closed contact manifold. As explained in \cite[p.~97]{SaMorse}, these points are leafwise fixed points of a Hamiltonian lift of the time-1-map to the symplectization.
\end{Rmk}

\begin{rmk}[$C^0$-small Hamiltonian diffeomorphism]\label{rmk:C0} Let $(M,\om)$ be a symplectic manifold. We denote by $\Ham_c(M,\om)$ the group of compactly supported Hamiltonian diffeomorphisms on $M$.%
\footnote{This is the set of Hamiltonian time-1-flows of compactly supported smooth functions on $[0,1]\x M$.}%
 Let $\UU$ be a neighbourhood of the identity in $\Ham_c(M,\om)$ \wrt the $C^0$ Whitney topology (= strong $C^0$ topology).%
\footnote{If $M$ is closed than this topology is the compact open topology.}
 
\textbf{Question:} Does there exist a neighbourhood of the identity $\VV$ in this topology, such that for every $\phi\in\VV$ there exists a compactly supported $H\in C^\infty\big([0,1]\x M,\R\big)$ satisfying $\phi_H^1=\phi$ and $\phi_H^t\in\UU$, for every $t$? 

If the answer to this question is yes then Theorem \ref{thm:leafwise} implies that given a closed coisotropic submanifold of $M$, every $\phi$ in some strong-$C^0$ neighbourhood of the identity in $\Ham_c(M,\om)$ has a leafwise fixed point. The answer is indeed yes for $M=\R^{2n}$ with the standard symplectic form. This follows from \cite[Lemma 3.2]{Se}. 
The answer is also affirmative if $M$ is a closed surface. This follows from the proof of \cite[Proposition 3.1]{Se} and from \cite[Lemma 3.2]{Se}. 
In general the question appears to be open.
\end{rmk}
To show that there exist $\UU$ and $\OO$ satisfying condition (\ref{thm:leafwise:Fix}) of Theorem \ref{thm:leafwise}, the idea is to construct a symplectic submanifold $\wt M$ of the manifold $M\x N$ that contains the diagonal embedding $\wt N$ of $N$ as a Lagrangian submanifold. This submanifold is constructed by using a smooth family of local slices in $N$ that are transverse to the isotropic distribution $TN^\om$. Such a family can be viewed as a substitute for the symplectic quotient of $N$. (This quotient is well-defined precisely if $N$ is regular.)

Because the restriction of the Hamiltonian flow $(\phi^t)$ to $N$ stays $C^0$-close to the inclusion $N\to M$, we can lift it from $M$ to $\wt M$. Intersection points of $\wt N$ with its image under the lifted time-1-flow correspond to leafwise fixed points of $\phi^1$. Using Weinstein's theorem we identify a neighbourhood of $\wt N$ in $\wt M$ with a neighbourhood of the zero section of the cotangent bundle $T^*\wt N$. Since the zero section is not displaceable in a Hamiltonian way, it will now follow that there exist $\UU$ and $\OO$ satisfying condition (\ref{thm:leafwise:Fix}) of Theorem \ref{thm:leafwise}. ($\OO$-nondegeneracy translates into transversality of the Lagrangian intersection in the cotangent bundle.)

This proof refines an approach from \cite{ZiLeafwise} that only works in the regular case. (See Remark \ref{rmk:regular} below.) 
\section{Linear holonomy of a foliation and nondegeneracy}\label{sec:nondeg}
In the second part of Theorem \ref{thm:leafwise}(\ref{thm:leafwise:Fix}) the pair $(\phi^1,N)$ is assumed to be $\OO$-nondegenerate. This is a weak version of the notion of nondegeneracy that was introduced in \cite{ZiLeafwise}.

The notion of $\OO$-nondegeneracy is based on the linear holonomy of a foliation. To explain this, let $M$ be a manifold and $\F$ a foliation on $M$, i.e., a maximal smooth atlas of foliation charts. We denote by $T\F\sub TM$ the tangent bundle of $\F$, and by
\[\pr:=\pr^\F:TM\to\NN\F:=TM/T\F\]
the canonical projection onto the normal bundle of $\F$. For $x\in M$ we write $T_x\F:=(T\F)_x$ and $\NN_x\F:=(\NN\F)_x$. ($T_x\F$ is the linear subspace of $T_xM$ tangent to the leaf through $x$.)

Let $F$ be a leaf of $\F$, $a\leq b$ and $x\in C^\infty\big([a,b],F\big)$. The linear holonomy of $\F$ along $x$ is a linear map $\hol^\F_x:\NN_{x(a)}\F\to \NN_{x(b)}\F$, whose definition is based on the following result.
\begin{prop}\label{prop:hol} Let $M,\F,F,a,b$ and $x$ be as above, $Y$ a manifold, and $y_0\in Y$. Then the following statements hold.
\begin{enui}\item\label{prop:hol:u} For every linear map $T: T_{y_0}Y\to T_{x(a)}M$ there exists a map $u\in C^\infty([a,b]\x Y,M)$ such that
\begin{eqnarray}\label{eq:u 0 x}&u(\cdot,y_0)=x,\quad \dd_tu(t,y)\in T_{u(t,y)}\F,\,\forall t\in[a,b],\,y\in Y,&\\
\label{eq:d u a y}&d(u(a,\cdot))(y_0)=T.& 
\end{eqnarray}
\item\label{prop:hol:u u'} Let $u,u'\in C^\infty([a,b]\x Y,M)$ be maps satisfying (\ref{eq:u 0 x}), such that
\begin{equation}\label{eq:u u' a}\pr d(u(a,\cdot))(y_0)=\pr d(u'(a,\cdot))(y_0).\end{equation}
Then $\pr d(u(b,\cdot))(y_0)=\pr d(u'(b,\cdot))(y_0)$. 
\end{enui}
\end{prop}
\begin{proof} See \cite[Proposition 2.1]{ZiLeafwise}.
\end{proof}
We define $Y:=\NN_{x(a)}\F$ and $y_0:=0$, and we canonically identify $T_0\big(\NN_{x(a)}\F\big)=\NN_{x(a)}\F$. We choose a linear map $T:\NN_{x(a)}\F\to T_{x(a)}M$, such that $\pr T=\id_{\NN_{x(a)}\F}$, and a map $u\in C^\infty\big([a,b]\x \NN_{x(a)}\F,M\big)$ such that (\ref{eq:u 0 x}) and (\ref{eq:d u a y}) hold. We define \emph{linear holonomy of $\F$ along $x$} to be the map
\begin{equation}\label{eq:hol F x}\hol^\F_x:=\pr d(u(b,\cdot))(0):\NN_{x(a)}\F(=T_0(\NN_{x(a)}\F))\to \NN_{x(b)}\F.\end{equation}
It follows from Proposition \ref{prop:hol} that this map is well-defined. Condition (\ref{eq:u 0 x}) in that proposition means that the $\hol^\F_x$ is obtained by sliding along the leaves of $\F$. The linear holonomy can be viewed as the linearization of the holonomy of a foliation as defined for example in Sec.~2.1 in the book \cite{MM}. For a given Riemannian metric on $M$, $\hol^\F_x$ corresponds to the holonomy of a Bott connection along $x$ (as defined e.g. in \cite[Lemma 6.1.5.]{CC}).

Let $(M,\om)$ be a symplectic manifold and $N\sub M$ a coisotropic submanifold. We denote by $\F^{N,\om}$ the isotropic foliation (=characteristic foliation) on $N$, and abbreviate
\begin{eqnarray*}&\hol^{N,\om}:=\hol^{\F^{N,\om}},&\\
&\pr:=\pr^{\F^{N,\om}}:TN\to \NN T\F^{N,\om}=TN/(TN)^\om.&
\end{eqnarray*}
Let $U\sub M$ be open and $\phi\in C^\infty(U,M)$. 
Let $F\sub N$ be an isotropic leaf and $x\in C^\infty([0,1],F)$ a path, such that
\begin{equation}\label{eq:x 0 U phi x}x(0)\in U,\quad\phi(x(0))=x(1).
\end{equation}
We call the triple $(\phi,N,x)$ \emph{nondegenerate} iff the only vector $v\in T_{x(0)}N\cap d\phi(x(0))^{-1}(T_{x(1)}N)$ satisfying
\begin{equation}\label{eq:hol x pr}\hol^{N,\om}_x\pr_{x(0)}v=\pr_{x(1)}d\phi(x(0))v,\end{equation}
is $v=0$. 

Let $\OO\sub N\x N$. 
We call $(\phi,N)$ (or simply $\phi$) \emph{$\OO$-nondegenerate} iff the following holds. Let $F\sub N$ be an isotropic leaf and $x\in C^\infty([0,1],F)$ a path, such that (\ref{eq:x 0 U phi x}) holds and 
\begin{equation}\label{eq:x 0 x t O}\big(x(0),x(t)\big)\in \OO,\,\forall t\in[0,1].
\end{equation}
Then $(\phi,N,x)$ is nondegenerate. This explains the $\OO$-nondegeneracy condition in Theorem \ref{thm:leafwise}.

\begin{Rmk} We call $(\phi,N)$ \emph{nondegenerate} iff it is $N\x N$-nondegenerate. This version of nondegeneracy was used in \cite{ZiLeafwise}. In the case $N=M$ this condition means that for every $x_0\in\Fix(\phi)$, $1$ is not an eigenvalue of $d\phi(x_0)$. Furthermore, in the case that $N$ is a connected Lagrangian and $\phi$ is submersive, $(\phi,N)$ is nondegenerate if and only if $N\pitchfork\phi^{-1}(N)$, i.e. $N$ and $\phi^{-1}(N)$ intersect transversely.
\end{Rmk}
\section{Proof of Theorem \ref{thm:leafwise}}\label{sec:proof}
The proof of Theorem \ref{thm:leafwise} is based on the following lemma. Let $(M,\om)$ be a symplectic manifold and $N\sub M$ a closed coisotropic submanifold. We equip $M\x M$ with the symplectic form $\om\oplus(-\om)$ and denote
\[2n:=\dim M,\quad k:=\codim N,\quad m:=n-k,\]
and by 
\[\iota_N:N\to M\]
the canonical inclusion.
\begin{lemma}[existence of a good symplectic submanifold of $M\x N$]\label{le:wt M} \begin{enui}
\item\label{le:wt M:exists} There exists a symplectic submanifold (without boundary) $\wt M$ of $M\x M$ of dimension $2n+2m$ that is contained in $M\x N$ and contains the diagonal
\[\wt N:=\big\{(x,x)\,\big|\,x\in N\big\}.\]
\item\label{le:wt M:wt N} Every such $\wt M$ contains $\wt N$ as a Lagrangian submanifold. 
\end{enui}
\end{lemma}
The idea of proof of this lemma is to choose an open neighbourhood $U\sub M$ of $N$, a retraction $r:U\to N$, and for each $y_0\in N$ a local slice $S_{y_0}\sub N$ through $y_0$ that is transverse to the isotropic foliation, and to define
\[``\wt M:=\big\{(x,y)\in U\x N\,\big|\,y\in S_{r(x)}\big\}".\]
The rank of the isotropic distribution equals $\codim N$. Hence $\dim(S_{y_0})=2n-2k=2m$, and therefore $\wt M$ will have the desired dimension $2n+2m$.
\begin{proof}[Proof of Lemma \ref{le:wt M}] \textbf{(\ref{le:wt M:exists}):} We will choose the submanifold $\wt M$ to be a subset of the image of the map $F$ whose existence is stated in the following claim.
\begin{claim}[existence of a good map to $M\x N$]\label{claim:hhat M} There exist a smooth manifold $\hhat M$ of dimension $2n+2m$, a compact subset $\hhat N\sub\hhat M$, and a smooth map
\[F:\hhat M\to M\x N\]
that maps $\hhat N$ bijectively to $\wt N$, such that the following holds. Let
\[\hhat x\in\hhat N.\]
We denote 
\begin{equation}\label{eq:wt V T}\wt V:=T_{F(\hhat x)}(M\x N),\,\,\wt\Om:=\big(\om\oplus(-\iota_N^*\om)\big)_{F(\hhat x)},\,\,\wt W:=dF(\hhat x)T_{\hhat x}\hhat M.\end{equation}
Then 
\begin{equation}\label{eq:wt W}\wt W\oplus\wt V^{\wt\Om}=\wt V.
\end{equation}
\end{claim}
(Here we denote by $\wt V^{\wt\Om}$ the $\wt\Om$-complement of $\wt V$.)
\begin{proof}[Proof of Claim \ref{claim:hhat M}] We will use the next claim, which roughly states that there exists a smooth family of local slices on $N$ that are transverse to the isotropic foliation. We denote
\[T_yN_\om:=T_yN/T_yN^\om,\,\forall y\in N,\quad TN_\om:=TN/TN^\om\]
\begin{claim}[local slices on $N$]\label{claim:slices} There exists a smooth map
\[f:TN_\om\to N,\]
such that, defining
\[f_y:=f(y,\cdot),\]
we have
\begin{eqnarray}\label{eq:f y 0}&f_y(0)=y,&\\
\label{eq:d f y}&df_y(0)(T_yN_\om)\oplus T_yN^\om=T_yN,&
\end{eqnarray}
for every $y\in N$. Here in (\ref{eq:d f y}) we canonically identified $T_0(T_yN_\om)$ with $T_yN_\om$.
\end{claim}
\begin{proof}[Proof of Claim \ref{claim:slices}] We choose a Riemannian metric $g$ on $N$ and denote by $\perp$ the orthogonal complement \wrt $g$. The map
\[TN\cont(TN^\om)^\perp\ni(y,v)\mapsto \big(y,v+T_yN^\om\big)\in TN_\om,\]
is an isomorphism of vector bundles over $N$. We denote by $\Phi$ the inverse of this map and by $\exp$ the exponential map \wrt $g$, and define
\[f:=\exp\circ\Phi:TN_\om\to N.\]
Let $y\in N$. Condition (\ref{eq:f y 0}) is satisfied. Furthermore, since $f_y=\exp_y\circ\Phi_y$, we have
\[df_y(0)T_yN_\om=d\exp_y(0)(T_yN^\om)^\perp=(T_yN^\om)^\perp.\]
Equality (\ref{eq:d f y}) follows. This proves Claim \ref{claim:slices}.
\end{proof}
We choose $f$ as in Claim \ref{claim:slices}, an open neighbourhood $U\sub M$ of $N$, and a smooth retraction
\[r:U\to N.\]
(This means that $r$ equals the identity on $N$.) We define
\[\hhat M:=r^*TN_\om=\big\{(x,\BAR v)\,\big|\,x\in U,\,\BAR v\in T_{r(x)}N_\om\big\}.\]
This is a submanifold of $M\x TN_\om$ of dimension $2n+2m$. We define
\begin{eqnarray*}&\hhat N:=\big\{(y,0_y)\,\big|\,y\in N\big\}\sub\hhat M,&\\
&F:\hhat M\to M\x N,\quad F(x,\BAR v):=\big(x,f(r(x),\BAR v)\big).&
\end{eqnarray*}
This map is smooth. Because of (\ref{eq:f y 0}) it maps $\hhat N$ bijectively to $\wt N$. 

To show \textbf{(\ref{eq:wt W})}, we define
\begin{eqnarray*}&R:U\to M\x N,\quad R(x):=(x,r(x)),&\\
&\hhat R:U\to\hhat M,\quad\hhat R(x):=(x,0_{r(x)}),&\\
&\iota:T_yN_\om\to\hhat M,\quad\iota(\BAR v):=(y,\BAR v).&
\end{eqnarray*}
Let $\hhat x=(y,0_y)\in\hhat N$. The map 
\[T_yM\x T_yN_\om\ni(v,\BAR v)\mapsto d\hhat R(y)v+d\iota(0)\BAR v\in T_{\hhat x}\hhat M\]
is an isomorphism. (Here we identified $T_0(T_yN_\om)=T_yN_\om$.) Since $R=F\circ\hhat R$, we have
\[dR(y)=dF(\hhat x)d\hhat R(y).\]
Since $F\circ\iota(\BAR v)=\big(y,f_y(\BAR v)\big)$, we have
\[dF(\hhat x)d\iota(0)=\big(0,df_y(0)\big):T_0(T_yN_\om)=T_yN_\om\to T_{F(\hhat x)}(M\x N).\]
It follows that
\begin{equation}\label{eq:wt W dF}\wt W=dF(\hhat x)T_{\hhat x}\hhat M=dR(y)T_yM+\{0\}\x df_y(0)T_yN_\om.\end{equation}
Since $\wt V^{\wt\Om}=\{0\}\x T_yN^\om$, it follows that
\begin{equation}\label{eq:wt W wt V}\wt W+\wt V^{\wt\Om}\cont dR(y)T_yM+\{0\}\x\big(df_y(0)T_yN_\om+T_yN^\om\big).\end{equation}
Using (\ref{eq:d f y}) and the fact that $dR(y)T_yM$ is the graph of a linear map (namely of $dr(y)$), the right hand side of (\ref{eq:wt W wt V}) equals $\wt V$. It follows that
\begin{equation}\label{eq:wt W wt V wt Om}\wt W+\wt V^{\wt\Om}=\wt V.\end{equation}
On the other hand, using (\ref{eq:wt W dF}), the facts that $dR(y)T_yM$ is a graph of a linear map, the sum (\ref{eq:d f y}) is direct, and that $\wt V^{\wt\Om}=\{0\}\x T_yN^\om$, it follows that 
\[\wt W\cap\wt V^{\wt\Om}=\{0\}.\]
(This also follows from (\ref{eq:wt W wt V wt Om}) and a dimension argument.) This proves (\ref{eq:wt W}) and completes the proof of Claim \ref{claim:hhat M}.
\end{proof}
We choose $\hhat M,\hhat N,F$ as in this claim. We show that the image $\wt M$ under $F$ of a suitable subset of $\hhat M$ is a symplectic submanifold of $M\x N$ that contains $\wt N=F(\hhat N)$, as desired. For this we show that $F$ is a symplectic immersion along $\hhat N$. 

Let $\hhat x\in\hhat N$. We define $\wt V,\wt\Om,\wt W$ as in (\ref{eq:wt V T}). 
\begin{claim}\label{claim:d F} The map $dF(\hhat x)$ is injective and $\wt W=dF(\hhat x)(T_{\hhat x}\hhat M)$ is a symplectic linear subspace of $\wt V$ \wrt $\wt\Om$.
\end{claim}
\begin{proof}[Proof of Claim \ref{claim:d F}] Condition (\ref{eq:wt W}) implies that
\begin{equation}\label{eq:dim wt W}\dim\wt W\geq\dim\wt V-\dim\wt V^{\wt\Om}=2n+2m.\end{equation}
Since $\dim T_{\hhat x}\hhat M=2n+2m$, it follows that $dF(\hhat x)$ is injective. This proves the first assertion.

Since $\wt V\sub\wt V^{\wt\Om\wt\Om}$, we have
\[\wt W^{\wt\Om}\sub\wt W^{\wt\Om}\cap\wt V^{\wt\Om\wt\Om}\sub\big(\wt W+\wt V^{\wt\Om}\big)^{\wt\Om}=\wt V^{\wt\Om}.\]
In the last step we used (\ref{eq:wt W}). (The inclusions ``$\cont$'' also hold, but we will not use them.) Using that the sum (\ref{eq:wt W}) is direct, it follows that
\[\wt W\cap\wt W^{\wt\Om}\sub\wt W\cap\wt V^{\wt\Om}=\{0\}.\]
Hence $\wt W$ is a symplectic linear subspace of $\wt V$. This proves Claim \ref{claim:d F}.
\end{proof}
Using Claim \ref{claim:d F} and injectivity of $F|_{\hhat N}$, by Lemma \ref{le:emb} below there exists an open neighbourhood $\hhat U\sub\hhat M$ of $\hhat N$, such that $F|_{\hhat U}$ is an embedding. We define
\[\wt M:=\big\{\wt x\in F(\hhat U)\,\big|\,\om\oplus(-\om)\textrm{ nondegenerate on }T_{\wt x}(F(\hhat U))\big\}.\]
By Lemma \ref{le:nondeg} below applied with the pullback of $\om\oplus(-\om)$ by the inclusion of $F(\hhat U)$ into $M\x M$, the set $\wt M$ is open in $F(\hhat U)$. Therefore $\wt M$ is a symplectic submanifold of $M\x M$ of dimension $2n+2m$. It is contained in $M\x N$. Claim \ref{claim:d F} and the fact $F(\hhat N)=\wt N$ imply that $\wt M$ contains $\wt N$. Hence $\wt M$ has the required properties. This proves part (\ref{le:wt M:exists}) of Lemma \ref{le:wt M}.\\

To prove \textbf{(\ref{le:wt M:wt N})}, let $x\in N$. The vectors in $T_{(x,x)}\wt N$ are those of the form $(v,v)$ with $v\in T_xN$. For $v,w\in T_xN$ we have 
\[\om\oplus(-\om)\left(\left(\begin{array}{c}
v\\
v
\end{array}\right),\left(\begin{array}{c}
w\\
w
\end{array}\right)\right)=0.\]
Hence $\wt N$ is an isotropic submanifold of $\wt M$. It is Lagrangian, since $\dim\wt M=2n+2m=2\dim\wt N$. This proves (\ref{le:wt M:wt N}) and thus Lemma \ref{le:wt M}.
\end{proof}
Let $\wt M\sub M\x N$ be a symplectic submanifold as in Lemma \ref{le:wt M}. The next lemma relates the Hamiltonian vector fields of a function on $M$ and its lift to $\wt M$. This will be used in the proof of Theorem \ref{thm:leafwise} to relate the corresponding Hamiltonian flows.

We denote by $\iota_{\wt M}$ the inclusion of $\wt M$ into $M\x M$, and define
\begin{equation}\label{eq:wt om}\wt\om:=\iota_{\wt M}^*\big(\om\oplus(-\om)\big).\end{equation}
This is a symplectic form on $\wt M$. Let $H\in C^\infty(M,\R)$. We denote by
\[\pr_1:\wt M\sub M\x N\to M,\quad\pr_2:\wt M\to N\]
the canonical projections, and define
\[\wt H:=H\circ\pr_1:\wt M\to\R.\]
\begin{lemma}[Hamiltonian vector field on symplectic submanifold $\wt M$]\label{le:d pr 1} We have
\begin{eqnarray}\label{eq:d pr 1}&d\pr_1X_{\wt H}=X_H\circ\pr_1,&\\
\label{eq:d pr 2}&d\pr_2X_{\wt H}\in TN^\om.&
\end{eqnarray}
\end{lemma}
\begin{proof}[Proof of Lemma \ref{le:d pr 1}] To see \textbf{(\ref{eq:d pr 1})}, observe that
\[\wt\om(X_{\wt H},\cdot)=d\wt H=(dH)\circ\pr_1d\pr_1=\om\big(X_H\circ\pr_1,d\pr_1\cdot\big).\]
Hence, for every $\wt x=(x,y)\in\wt M$ and $v\in T_xM$, we have
\[\om\big(X_H(x),v\big)=\wt\om\big(X_{\wt H}(\wt x),(v,0)\big)=\om\big(d\pr_1X_{\wt H}(\wt x),v\big).\]
The second equality follows from the definition (\ref{eq:wt om}) of $\wt\om$. It follows that
\[X_H(x)=d\pr_1X_{\wt H}(\wt x).\]
This proves (\ref{eq:d pr 1}).

We prove \textbf{(\ref{eq:d pr 2})}. Let $\wt x=(x,y)\in\wt M$. We define
\begin{eqnarray*}&H_y:=\big\{w\in T_yN\,\big|\,(0,w)\in T_{\wt x}\wt M\big\},&\\
&\wt V:=T_{\wt x}(M\x N),\quad\wt\Om:=\big(\om\oplus(-\iota_N^*\om)\big)_{\wt x}.&
\end{eqnarray*}
Since $T_{\wt x}\wt M$ is a symplectic subspace of $\wt V$ of maximal dimension $2n+2m$, we have
\[\wt V=T_{\wt x}\wt M+\wt V^{\wt\Om}.\]
Combining this with the fact $\wt V^{\wt\Om}=\{0\}\x T_yN^\om$, it follows that 
\begin{equation}\label{eq:T y N}T_yN\sub H_y+T_yN^\om.\end{equation}
For every $w\in H_y$ we have
\begin{align*}0&=dHd\pr_1(\wt x)(0,w)\\
&=d\wt H(\wt x)(0,w)\\
&=\wt\om\big(X_{\wt H}(\wt x),(0,w)\big)\\
&=-\om\big(d\pr_2 X_{\wt H}(\wt x),w\big).
\end{align*}
Because of (\ref{eq:T y N}) it follows that $\om\big(d\pr_2 X_{\wt H}(\wt x),w\big)=0$, for every $w\in T_yN$. This shows (\ref{eq:d pr 2}) and completes the proof of Lemma \ref{le:d pr 1}.
\end{proof}
The next lemma produces leafwise fixed points for a $\phi\in\Ham(M,\om)$ out of Lagrangian intersection points of $\wt N$ and its translation under a lift of $\phi$. 
Let
\[H\in C^\infty\big([0,1]\x M,\R\big).\]
We define
\begin{equation}\label{eq:wt H}\wt H_t:=H_t\circ\pr_1:\wt M\to\R,\quad\wt H:=(\wt H_t)_{t\in[0,1]}.
\end{equation}
\begin{lemma}[Lagrangian intersection points and leafwise fixed points]\label{le:wt x} 
For every
\[\wt x\in\wt N\cap(\phi_{\wt H}^1)^{-1}(\wt N)\]
we have
\[x:=\pr_1(\wt x)\in\Fix(\phi_H^1,N).\]
\end{lemma}
\begin{Rmk} The set $(\phi_{\wt H}^1)^{-1}(\wt N)$ is contained in the domain of the flow $\phi_{\wt H}^1$, which is an open subset of $\wt M$.
\end{Rmk}
\begin{proof}[Proof of Lemma \ref{le:wt x}] Since $\wt x\in\dom\big(\phi_{\wt H}^1\big)$, equality (\ref{eq:d pr 1}) in Lemma \ref{le:d pr 1} implies that for every $t\in[0,1]$, we have $x\in\dom(\phi_H^t)$ and
\begin{equation}\label{eq:phi H t x}\phi_H^t(x)=\pr_1\circ\phi_{\wt H}^t(\wt x).\end{equation}
We denote
\[y:[0,1]\to N,\quad y(t):=\pr_2\circ\phi_{\wt H}^t(\wt x).\]
Since $\wt x\in\wt N$, we have
\[y(0)=\pr_2(\wt x)=\pr_1(\wt x)=x\in N.\]
By statement (\ref{eq:d pr 2}) in Lemma \ref{le:d pr 1}, we have
\[\dot y(t)\in T_{y(t)}N^\om,\quad\forall t\in[0,1].\]
It follows that 
\begin{equation}\label{eq:y leaf}y([0,1])\sub\textrm{ isotropic leaf through }x.
\end{equation}
Equality (\ref{eq:phi H t x}) and our assumption that $\wt x\in(\phi_{\wt H}^1)^{-1}(\wt N)$ imply that 
\[\phi_H^1(x)=\pr_1\circ\phi_{\wt H}^1(\wt x)=\pr_2\circ\phi_{\wt H}^1(\wt x)=y(1).\]
Since $y(1)$ lies in the isotropic leaf of $x$, it follows that $x$ is a leafwise fixed point for $\phi_H^1$. This proves Lemma \ref{le:wt x}.
\end{proof}
For the proof of Theorem \ref{thm:leafwise} we also need the following. 

\begin{lemma}[nondegeneracy and transversality]\label{le:N phi} Let $U\sub M$ and $\wt U\sub\wt M$ be open and 
$\phi\in C^\infty\big([0,1]\x U,M\big)$, $\wt\phi\in C^\infty\big([0,1]\x\wt U,\wt M\big)$ be maps, such that
\begin{eqnarray}\label{eq:wt phi 0}&\wt\phi_0=\textrm{inclusion: }\wt U\to\wt M,&\\
\label{eq:pr 1 wt U sub U}&\pr_1(\wt U)\sub U,&\\
\label{eq:pr 1 wt phi t}&\pr_1\circ\wt\phi_t=\phi_t\circ\pr_1|_{\wt U},\quad\forall t\in[0,1],&\\
\label{eq:pr 2 wt phi t}&\pr_2\circ\wt\phi_t(\wt x)\in\textrm{ isotropic leaf through }\pr_2(\wt x),\,\forall\wt x\in\wt U,\,t\in[0,1].&
\end{eqnarray}
Let $\wt x\in\wt N\cap\wt\phi_1^{-1}(\wt N)$. 
We define 
\[y:[0,1]\to N,\quad y(t):=\pr_2\circ\wt\phi_t(\wt x).\]
The triple $\big(\phi_1,N,y\big)$ is nondegenerate if and only if
\[T_{\wt x}\wt N\cap d\wt\phi_1(\wt x)^{-1}\left(T_{\wt\phi_1(\wt x)}\wt N\right)=\{0\}.\]
\end{lemma}
\begin{Rmks} 
\begin{itemize}
\item Because of (\ref{eq:pr 1 wt U sub U}) the condition (\ref{eq:pr 1 wt phi t}) makes sense.
\item It follows from (\ref{eq:wt phi 0}) that
\[y(0)=\pr_2\circ\wt\phi_0(\wt x)=\pr_2(\wt x).\]
Since $\wt x\in\wt N$, it follows that
\begin{equation}\label{eq:wt x y 0}\wt x=\big(y(0),y(0)\big).\end{equation}
Since $\wt\phi_1(\wt x)\in\wt N$, we have $\wt\phi_1(\wt x)=\big(y(1),y(1)\big)$, and therefore by (\ref{eq:pr 1 wt phi t},\ref{eq:wt x y 0}),
\begin{equation}\label{eq:y 1 phi 1 y 0}y(1)=\pr_1\circ\wt\phi_1(\wt x)=\phi_1\circ\pr_1(\wt x)=\phi_1\circ y(0).\end{equation}
Hence using (\ref{eq:pr 2 wt phi t}), the nondegeneracy condition for $(\phi_1,N,y)$ makes sense.

\item The condition (\ref{eq:pr 1 wt phi t}) implies that
\begin{equation}\label{eq:d pr 1 d wt phi 1}d\pr_1d\wt\phi_1(\wt x)=d\phi_1d\pr_1(\wt x)\quad\textrm{on}\quad T_{\wt x}\wt M.\end{equation}
\end{itemize}
\end{Rmks}
\begin{proof}[Proof of Lemma \ref{le:N phi}]\label{proof:le:N phi} \setcounter{claim}{0} We denote
\[x:=\pr_1(\wt x)=y(0),\quad\hol_y:=\hol^{N,\om}_y,\]
by
\[\pr:TN\to TN/(TN)^\om\]
the canonical projection, and
\[X:=\Big\{v\in T_xN\cap d\phi_1(x)^{-1}\left(T_{\phi_1(x)}N\right)\,\Big|\,\hol_y\pr v=\pr d\phi_1(x)v\Big\}.\]
The triple $\big(\phi_1,N,y\big)$ is nondegenerate if and only if $X=\{0\}$. The statement of Lemma \ref{le:N phi} therefore follows from the next claim.
\begin{claim}\label{claim:d pr 1 wt x} The map
\begin{equation}\label{eq:d pr 1 wt x}d\pr_1(\wt x):T_{\wt x}\wt N\cap d\wt\phi_1(\wt x)^{-1}\left(T_{\wt\phi_1(\wt x)}\wt N\right)\to X\end{equation}
is well-defined and bijective.
\end{claim}
\begin{pf}[Proof of Claim \ref{claim:d pr 1 wt x}] Since $\wt N=\big\{(y_0,y_0)\,\big|\,y_0\in N\big\}$, the maps
\[d\pr_1(\wt x):T_{\wt x}\wt N\to T_xN,\quad\iota:T_xN\to T_{\wt x}\wt N,\,\iota(v)=(v,v),\]
are well-defined and inverses of each other. Let
\[\wt v\in T_{\wt x}\wt N.\]
We define
\[v:=d\pr_1(\wt x)\wt v\in T_xN,\quad\wt w:=d\wt\phi_1(\wt x)\wt v.\]
In order to prove Claim \ref{claim:d pr 1 wt x}, it suffices to show that
\begin{equation}\label{eq:d wt phi 1 wt x wt v}\wt w\in T_{\wt\phi_1(\wt x)}\wt N\iff d\phi_1(x)v\in T_{\phi_1(x)}N,\,\hol_y\pr v=\pr d\phi_1(x)v.
\end{equation}
For this we need the following claim. We write
\[w_i:=d\pr_i\big(\wt\phi_1(\wt x)\big)\wt w.\]
\begin{claim}\label{claim:hol} We have
\[\hol_y\pr v=\pr w_2.\]
\end{claim}
\begin{proof}[Proof of Claim \ref{claim:hol}] Since $\wt x=(x,x)\in\wt\phi_1^{-1}(\wt N)\sub\wt U$, there exists a path $z\in C^\infty(\R,N)$, such that
\[z(0)=x,\quad\dot z(0)=v,\quad(z,z)(\R)\sub\wt U.\]
We define
\[u:\R\x[0,1]\to N,\quad u(s,t):=\pr_2\circ\wt\phi_t\circ(z,z)(s).\]
We have
\[\dd_su(0,0)=\dot z(0)=v.\]
It follows from (\ref{eq:pr 2 wt phi t}) that $u(s,t)$ lies in the isotropic leaf through $z(s)$, for every $s,t$. Therefore, by the definition of the linear holonomy, we have 
\begin{align*}\hol_y\pr v&=\pr\dd_su(0,1)\\
&=\pr d\pr_2\big(\wt\phi_1(\wt x)\big)d\wt\phi_1(x,x)(v,v)\\
&=\pr w_2.
\end{align*}
This proves Claim \ref{claim:hol}.
\end{proof}
The equalities (\ref{eq:d pr 1 d wt phi 1}) and $v=d\pr_1(\wt x)\wt v$ imply that 
\begin{equation}\label{eq:w 1}w_1=d\pr_1\big(\wt\phi_1(\wt x)\big)\wt w=d\phi_1(x)v.\end{equation}
We show the \textbf{implication ``$\then$'' in (\ref{eq:d wt phi 1 wt x wt v})}. Assume that $\wt w\in T_{\wt\phi_1(\wt x)}\wt N$. Then we have
\begin{equation}\label{eq:d pr 2 d wt phi 1}w_2=w_1\in T_{\pr_1\circ\wt\phi_1(\wt x)}N=T_{\phi_1(x)}N.\end{equation}
Using (\ref{eq:w 1}), it follows that
\[d\phi_1(x)v\in T_{\phi_1(x)}N.\]
Combining Claim \ref{claim:hol} and (\ref{eq:d pr 2 d wt phi 1},\ref{eq:w 1}), 
we obtain
\[\hol_y\pr v=\pr d\phi_1(x)v.\]
This shows the implication ``$\then$'' in (\ref{eq:d wt phi 1 wt x wt v}).\\

We show the \textbf{implication ``$\follows$''}. Assume that
\[d\phi_1(x)v\in T_{\phi_1(x)}N,\quad\hol_y\pr v=\pr d\phi_1(x)v.\]
By (\ref{eq:w 1}) it follows that
\begin{equation}\label{eq:pr d pr 1 wt x}\pr w_1=\pr d\phi_1(x)v=\hol_y\pr v.\end{equation}
Using Claim \ref{claim:hol}, it follows that
\begin{equation}\label{eq:pr w 1}\pr w_1=\pr w_2.\end{equation}
Therefore, for every $\wt w'=(w',w')\in T_{\wt\phi_1(\wt x)}\wt N$, we have
\[\wt\om(\wt w,\wt w')=\om(w_1,w')-\om(w_2,w')=\om\big(w_1-w_2,w'\big)=0.\]
Here we used (\ref{eq:pr w 1}) in the last equality. It follows that
\[\wt w=(w_1,w_2)\in T_{\wt\phi_1(\wt x)}\wt N^{\wt\om}=T_{\wt\phi_1(\wt x)}\wt N.\]
Here we used that $\wt N=\big\{(x,x)\,\big|\,x\in N\big\}$ is a Lagrangian submanifold of $\wt M$. This proves the implication ``$\follows$'' in (\ref{eq:d wt phi 1 wt x wt v}). This completes the proof of (\ref{eq:d wt phi 1 wt x wt v}) and therefore of Claim \ref{claim:d pr 1 wt x} and of Lemma \ref{le:N phi}.
\end{pf}
\end{proof}
The proof of part (\ref{thm:leafwise:dense}) of the statement of Theorem \ref{thm:leafwise} is based on the following result.

\begin{prop}\label{prop:transv} Let $A,X,Z$ be (smooth) manifolds (without boundary), $Y\sub Z$ a submanifold, and $f\in C^\infty(A\x X,Z)$, be a map, such that $f\pitchfork Y$, i.e., $f$ is transverse to $Y$. Then the set
\[\big\{a\in A\,\big|\,f(a,\cdot)\pitchfork Y\big\}\] 
is residual in $A$, i.e., it contains a countable intersection of open and dense sets.
\end{prop} 
\begin{proof} \cite[2.7.~Theorem]{Hi}.
\end{proof}
For the proof of part (\ref{thm:leafwise:dense}) of Theorem \ref{thm:leafwise} we also need the following lemma and remark. Let $(M,\om)$ be a symplectic manifold. 


\begin{lemma}\label{le:H} Let $K\sub M$ be compact and $K'\sub M$ be a compact neighbourhood of $K$. There exists $\ell\in\N$ and a function $H\in C^\infty\big(B^\ell_1\x M,\R\big)$, such that for every $a\in B^\ell_1$ (the open unit ball in $\R^\ell$), the function $H_a:=H(a,\cdot)$ has support in $K'$, 
the map $a\mapsto H_a$ is the restriction of a linear map to $B^\ell_1$, 
\begin{eqnarray}\label{eq:surj}&d\big(a\mapsto\phi_{H_a}^1(x)\big)(a):T_aB^\ell_1=\R^\ell\to T_{\phi_{H_a}^1(x)}M\textrm{ is surjective,}&\\
\nn&\forall a\in B^\ell_1,\,x\in K,&\\
\label{eq:H 0}&\textrm{and }H_0\const0.&
\end{eqnarray}
\end{lemma}


\begin{proof}[Proof of Lemma \ref{le:H}]\setcounter{claim}{0}  Let $x\in M$. We denote $2n:=\dim M$.

\begin{claim}\label{claim:F x} There exists a linear map
\[F^x:\R^{2n}\to C^\infty(M,\R),\]
such that $\R^{2n}\ni w\mapsto X_{F^x(w)}(x)\in T_xM$ is bijective, and $F^x(w)$ has support in $K'$, for every $w\in\R^{2n}$. 
\end{claim}

\begin{proof}[Proof of Claim \ref{claim:F x}] Let $v_1,\ldots,v_{2n}$ be a basis of $T_xM$. Let $i\in\{1,\ldots,2n\}$. It follows from an argument in a Darboux chart that there exists a function $F_i\in C^\infty(M,\R)$ with support in $K'$, such that
\[X_{F_i}(x)=v_i.\]
We define
\[F^x:\R^{2n}\to C^\infty(M,\R),\quad F^x(w)(y):=\sum_{i=1}^{2n}w^iF_i(y),\]
where $w^i$ denotes the $i$-th coordinate of $w$ \wrt the standard basis. This map has the required properties. This proves Claim \ref{claim:F x}.
\end{proof}
We choose $F^x$ as in this claim and write $F^x_w:=F^x(w)$. The set
\begin{equation}\label{eq:U x}U_x:=\big\{y\in M\,\big|\,\R^{2n}\ni w\mapsto X_{F^x_w}(y)\in T_yM\textrm{ is bijective}\big\}\end{equation}
is open and contains $x$. Since $K$ is compact, there exists a finite set $S\sub K$, such that
\[\bigcup_{x\in S}U_x=K.\]
We define
\[F:(\R^{2n})^S\to C^\infty(M,\R),\quad F_a:=F(a):=\sum_{x\in S}F^x_{a(x)}.\]
\begin{claim}\label{claim:d} For every $y\in K$ the map
\[d\big(a\mapsto\phi_{F_a}^1(y)\big)(0):T_0(\R^{2n})^S=(\R^{2n})^S\to T_yM\]
is surjective.
\end{claim}
\begin{proof}[Proof of Claim \ref{claim:d}] Let $v\in T_yM$. We choose $x\in K$ such that $y\in U_x$. By (\ref{eq:U x}) there exists $w\in\R^{2n}$, such that
\begin{equation}\label{eq:X F x w}X_{F^x_w}(y)=v.\end{equation}
We define
\[\hhat a:S\to\R^{2n},\quad\hhat a(x'):=\left\{\begin{array}{ll}
w,&\textrm{if }x'=x,\\
0,&\textrm{otherwise.}
\end{array}\right.\]
Let $t\in[0,1]$. We have $F_{t\hhat a}=tF_{\hhat a}$, hence $X_{F_{t\hhat a}}=tX_{F_{\hhat a}}$, and therefore,
\[\phi_{F_{t\hhat a}}^1=\phi_{F_{\hhat a}}^t.\]
Furthermore, we have
\[F_{\hhat a}=\sum_{x'\in S}F_{a(x')}^{x'}=F^x_w.\]
It follows that
%
%
\begin{align*}d\big(a\mapsto\phi_{F_a}^1(y)\big)(0)\hhat a&=\left.\frac d{dt}\right|_{t=0}\phi_{F_{t\hhat a}}^1(y)\\
&=\left.\frac d{dt}\right|_{t=0}\phi_{F_{\hhat a}}^t(y)\\
&=X_{F_{\hhat a}}(y)\\
&=X_{F^x_w}(y)\\
&=v.
\end{align*}
Here in the last step we used (\ref{eq:X F x w}). This proves Claim \ref{claim:d}.
\end{proof}

We define $\ell:=2n|S|$, choose a linear isomorphism $T:\R^\ell\to(\R^{2n})^S$, and define
\[H:\R^\ell\x M\to\R,\quad H(a,y):=\big(F\circ T(a)\big)(y).\]
By Claim \ref{claim:d}, restricting $H$ to the product of a ball in $\R^\ell$ and $M$, and rescaling, we may assume \Wlog that (\ref{eq:surj}) holds. This function has the required properties. This proves Lemma \ref{le:H}.
\end{proof}
\begin{rmk}\label{rmk:hol} Let $(M,\F)$ be a foliated manifold, $F$ a leaf of $\F$, and $x,y\in C^\infty([0,1],F)$, such that $x(i)=y(i)$, for $i=0,1$. Assume that there exists a surjective foliation chart $U\sub M\to\R^\ell\x\R^k$, such that $x([0,1]),y([0,1])\sub U$. Then 
\[\hol^\F_x=\hol^\F_y.\]
To see this, consider the case $M=\R^\ell\x\R^k$ with the standard foliation. Then we have
\[\hol^\F_x=\id:\nu_{x(0)}\F=\R^\ell\x\R^k/\{0\}\x\R^k\to\nu_{x(1)}\F=\R^\ell\x\R^k/\{0\}\x\R^k,\]
hence the statement holds. The general situation can be reduced to this case.
\end{rmk}
We are now ready for the proof of the main result.
\begin{proof}[Proof of Theorem \ref{thm:leafwise}]\label{proof:thm:leafwise}\setcounter{claim}{0} We choose a submanifold $\wt M$ as in Lemma \ref{le:wt M}. Shrinking $\wt M$ if necessary, by Weinstein's neighbourhood theorem we may assume \Wlog that there exists a symplectomorphism between $\wt M$ and an open neighbourhood of the zero section in $T^*\wt N$, that is the identity on $\wt N$. 

In order to construct the desired $C^0$-neighbourhood of the inclusion $N\to M$, we need to ensure that the lifted Hamiltonian flow $\phi_{\wt H}^t$ is defined on the subset $\wt N$ of $\wt M$. For this we need to control $d\pr_2X_{\wt H_t}$, where $\pr_2:\wt M\sub M\x N\to N$ denotes the projection. By Lemma \ref{le:d pr 1} $d\pr_2X_{\wt H_t}$ lies in the isotropic distribution. Using foliation charts, we obtain local charts for $\wt M$ in which this ``isotropic part'' of the lifted Hamiltonian vector field disappears. (See (\ref{eq:d wt phi}) in the proof of Claim \ref{claim:phi wt H} below.) This will make the lifted Hamiltonian flow well-defined on $\wt N$. 

To construct these charts, we need the following. We denote by
\[\pi:\R^{2m}\x\R^k\to\R^{2m}\]
the canonical projection onto the first factor. Let
\[y\in N.\]
%
%
\begin{claim}\label{claim:V y} There exists an open neighbourhood $V$ of $y$ in $N$ and a surjective 
foliation chart $\psi:V\to\R^{2m}\x\R^k$, satisfying
\begin{equation}\label{eq:z z'}(z,z')\in\wt M\cap(V\x V)\textrm{ such that }\pi\circ\psi(z)=\pi\circ\psi(z')\then z=z'.
\end{equation}
\end{claim}
\begin{proof}[Proof of Claim \ref{claim:V y}] We choose an open neighbourhood $V_0$ of $y$ in $N$ and a foliation chart $\psi_0:V_0\to\R^{2m}\x\R^k$, satisfying $\psi_0(y)=0$. For every subset $V\sub V_0$ we define
\[\wt X_V:=(\psi_0\x\psi_0)\big(\wt M\cap(V\x V)\big)\sub\R^{4m+2k}.\]
We define the map
\[f:\wt X_{V_0}\x\R^k\to\R^{4m+k}\x\R^k,\quad f(\wt x,w):=\wt x+(0,w).\]
We show that $\wt X_{V_0}$ is a submanifold of $\R^{4m+2k}$ and that $f$ is an immersion. Let $\wt x=(z,z')\in\wt M\cap(V_0\x V_0)$. Since $\wt M$ is a symplectic submanifold of $M\x N$, we have
\[T_{\wt x}\wt M\cap\left(\{0\}\x T_{z'}N^\om\right)=\{0\}.\]
This intersection is transverse, since
\[\dim(\wt M)+\dim\left(\{0\}\x T_{z'}N^\om\right)=\dim(M\x N).\] 
Hence we have $T_{\wt x}\wt M\pitchfork T_{\wt x}(N\x N)$. It follows that $\wt M\cap(N\x N)$ is a submanifold of $N\x N$, and that
\[T_{\wt x}\big(\wt M\cap(N\x N)\big)\cap\left(\{0\}\x T_{z'}N^\om\right)=\{0\}.\]
This implies that $\wt X_{V_0}$ is a submanifold of $\R^{4m+2k}$ and that 
\[T_{\left(\psi_0\x\psi_0\right)(\wt x)}\wt X_{V_0}\cap\left(\{0\}\x\R^k\right)=\{0\}.\]
It follows that $f$ is an immersion, as claimed.\\

Since $f$ is immersive, there exist neighbourhoods $V\sub V_0$ of $y$ in $N$ and $W$ of $0$ in $\R^k$, such that the restriction of $f$ to $\wt X_V\x W$ is injective. Shrinking $V$, we may assume \Wlog that
\[\psi_0(V)=B^{2m}_r\x B^k_r,\quad B^k_{2r}\sub W,\]
for some $r>0$. Here $B^k_r$ denotes the open ball in $\R^k$ about 0 with radius $r$. For every $\ell=0,1,\ldots$ we choose a diffeomorphism $\chi_\ell:B^\ell_r\to\R^\ell$. We define
\[\psi:=(\chi_{2m}\x\chi_k)\circ\psi_0|_V:V\to\R^{2m}\x\R^k.\]
This is a surjective foliation chart. We show that (\ref{eq:z z'}) holds. Let $(z,z')\in\wt M\cap(V\x V)$ be such that $\pi\circ\psi(z)=\pi\circ\psi(z')$. We denote by $\pi':\R^{2m}\x\R^k\to\R^k$ the canonical projection onto the second factor. We define
\[w:=\pi'\big(\psi_0(z')-\psi_0(z)\big).\]
Since $\pi'\circ\psi_0(z),\pi'\circ\psi_0(z')\in B^k_r$, we have $w\in B^k_{2r}$. The equality $\pi\circ\psi(z)=\pi\circ\psi(z')$ implies that $\pi\circ\psi_0(z)=\pi\circ\psi_0(z')$. Hence we have
\begin{equation}\label{eq:psi 0 z}\big(\psi_0(z),\psi_0(z)\big)+(0,w)=\big(\psi_0(z),\psi_0(z')\big),\end{equation}
where $0\in\R^{2m+k}\x\R^{2m}$. Since $(z,z)\in V\x V\sub\wt N\sub\wt M$, we have $\big(\psi_0(z),\psi_0(z)\big)\in\wt X_V$. Furthermore, we have $\big(\psi_0(z),\psi_0(z')\big)\in\wt X_V$. Since the restriction of $f$ to $\wt X_V\x B^k_{2r}$ is injective, using (\ref{eq:psi 0 z}), it follows that $\psi_0(z)=\psi_0(z')$ (and $w=0$), hence $z=z'$. Hence (\ref{eq:z z'}) holds. Therefore, $\psi$ has the desired properties. This proves Claim \ref{claim:V y}.
\end{proof}
We choose $V=V_y$ and $\psi=\psi_y$ as in Claim \ref{claim:V y}. We define
\begin{eqnarray}\label{eq:wt phi}&\wt\psi_y:=\id_M\x(\pi\circ\psi_y):\wt M\cap(M\x V_y)\to M\x\R^{2m},&\\
\nn&\wt x:=(y,y).&
\end{eqnarray}
\begin{claim}\label{claim:d wt phi} The derivative
\[\wt\Psi:=d\wt\psi_y(\wt x):T_{\wt x}\wt M\to T_yM\x\R^{2m}\]
is invertible.
\end{claim}
\begin{proof}[Proof of Claim \ref{claim:d wt phi}] We denote
\[\Psi:=d\pi d\psi_y(y):T_yN\to\R^{2m}.\]
We have
\[\wt\Psi=\id\x\Psi:T_{\wt x}\wt M\sub T_yM\x T_yN\to T_yM\x\R^{2m}.\]
Let $\wt v=(v,v')\in\ker\wt\Psi$. Then $v=0$ and $v'\in\ker\Psi=T_yN^\om$, hence $\wt v=(0,v')\in T_{\wt x}\wt M^{\wt\om}$. Since $\wt M$ is symplectic, it follows that $\wt v=0$. Therefore $\wt\Psi$ is injective. Since the domain and target of $\wt\Psi$ both have dimension $2n+2m$, it follows that $\wt\Psi$ is an isomorphism. This proves Claim \ref{claim:d wt phi}.
\end{proof}
Using Claim \ref{claim:d wt phi}, by the Inverse Function Theorem there exist open neighbourhoods $U_y$ of $y$ in $M$ and $W_y\sub\R^{2m}$ of $\pi\circ\psi_y(y)$, such that the restriction
\begin{equation}\label{eq:wt psi wt U}\wt\psi_y:\wt U_y:=\wt\psi_y^{-1}(U_y\x W_y)\to U_y\x W_y\end{equation}
is a diffeomorphism. (Note that
\[\wt U_y=\big(U_y\x(\pi\circ\psi_y)^{-1}(W_y)\big)\cap\wt M.)\]
The map $\wt\psi_y$ plays the role of a local ``chart'' for $\wt M$. (The target of this map is an open subset of $M\x\R^{2m}$, hence $\wt\psi_y$ is not a chart in the strict sense.) We will show that it has the desired property, see (\ref{eq:d wt phi}) below.

Shrinking $U_y$, we may assume \Wlog that $U_y\cap N\sub V_y$. 
We choose a compact neighbourhood $K_y$ of $y$ in $U_y\cap N$, such that $\pi\circ\psi_y(K_y)\sub W_y$. 
The last condition makes sense, since $U_y\cap N\sub V_y$ and $\psi_y$ is defined on $V_y$. Since $N$ is compact, there exist $\ell\in\N$ and points $y_1,\ldots,y_\ell\in N$, such that 
\begin{equation}\label{eq:N sub bigcup}N\sub\bigcup_{i=1,\ldots,\ell}\Int K_{y_i}.\end{equation}
(Here $\Int X$ denotes the interior of $X$ in $N$.) We denote
\begin{eqnarray*}&K_i:=K_{y_i},\quad V_i:=V_{y_i},\quad U_i:=U_{y_i},\quad W_i:=W_{y_i},&\\
&\wt U_i:=\wt U_{y_i},\quad\psi_i:=\psi_{y_i},\quad\wt\psi_i:=\wt\psi_{y_i}.&
\end{eqnarray*}
We define
\[\UU:=\big\{\phi\in C(N,M)\,\big|\,\phi(K_i)\sub U_i,\,\forall i=1,\ldots,\ell\big\}.\]
This is a neighbourhood of the inclusion $\iota_N:N\to M$ in the compact open topology (= $C^0$-topology). We define
\begin{equation}\label{eq:OO}\OO:=\bigcup_{i=1,\ldots,\ell}\left(\Int K_i\wo\bigcup_{j=1,\ldots,i-1}\Int K_j\right)\x V_i.
\end{equation}
Here we use the convention that for $i=1$ the union $\bigcup_{j=1,\ldots,i-1}X_j$ is empty. 
\begin{claim}\label{claim:OO} The set $\OO$ is a (possibly nonopen) neighbourhood of the diagonal $\wt N$ in $N\x N$.
\end{claim}
\begin{proof}[Proof of Claim \ref{claim:OO}] Let $y\in N$. Let $i_0$ be the smallest index, such that $y\in\Int K_{i_0}$. (Here we use (\ref{eq:N sub bigcup}).) We define
\[V':=\Int K_{i_0}\wo\bigcup_{j\in\{1,\ldots,\ell\}:\,y\not\in K_j}K_j,\quad V'':=\bigcap_{i\in\{1,\ldots,\ell\}:\,y\in K_i}V_i.\]
$V'$ and $V''$ are open neighbourhoods of $y$ in $N$. (Here we use that $K_i\sub U_i\cap N\sub V_i$.) It follows that $V'\x V''$ is an open neighbourhood of $(y,y)$ in $N\x N$. Applying Remark \ref{rmk:X Y} from the appendix with $J:=\{1,\ldots,\ell\}$, $I:=\big\{i\in J\,\big|\,y\in K_i\big\}$, $A_i:=\Int K_i$, and $B_i:=K_i$, it follows that $V'\x V''\sub\OO$. The statement of Claim \ref{claim:OO} follows.
\end{proof}

\textbf{We show that $\UU$ and $\OO$ satisfy condition (\ref{thm:leafwise:Fix}) of Theorem \ref{thm:leafwise}.} The next claim will be used to show that given a Hamiltonian flow on $M$ as in (\ref{thm:leafwise:Fix}), $\wt N$ intersects its translate under the lifted flow on $\wt M$. 

Let $H\in C^\infty\big([0,1]\x M,\R\big)$ be a function whose Hamiltonian flow is globally defined, surjective, and satisfies $\phi^t_H|_N\in\UU$, for every $t\in[0,1]$. We define
\[\wt H_t:=H_t\circ\pr_1:\wt M\to\R.\]
Let $\wt x_0=(y_0,y_0)\in\wt N$.
\begin{claim}\label{claim:phi wt H} We have
\[\wt x_0\in\dom\big(\phi_{\wt H}^t\big),\quad\phi_{\wt H}^t(\wt x_0)\in\wt U_i,\]
for every $t\in[0,1]$ and every $i\in\{1,\ldots,\ell\}$, such that $y_0\in\Int K_i$. 
%
%
\end{claim}
\begin{proof}[Proof of Claim \ref{claim:phi wt H}] Let $i$ be such that $y_0\in\Int K_i$. We abbreviate
\begin{eqnarray*}&K:=K_i,\quad V:=V_i,\quad U:=U_i,\quad W:=W_i,&\\
&\wt U:=\wt U_i,\quad\psi:=\psi_i,\quad\wt\psi:=\wt\psi_i.&
\end{eqnarray*}
Recall that $\pi:\R^{2m}\x\R^k\to\R^{2m}$ denotes the canonical projection. Let $t\in[0,1]$ and $\wt x=(x,y)\in\wt U$. Since $\psi$ is a foliation chart, the derivative $d(\pi\circ\psi)(y)=d\pi(\psi(y))d\psi(y)$ vanishes on $T_yN^\om$. Hence statement (\ref{eq:d pr 2}) in Lemma \ref{le:d pr 1} implies that
\[d\big(\pi\circ\psi\circ\pr_2\big)X_{\wt H_t}(\wt x)=d(\pi\circ\psi)(y)d\pr_2(\wt x)X_{\wt H_t}(\wt x)=0.\]
Using equality (\ref{eq:d pr 1}) in Lemma \ref{le:d pr 1}, it follows that
\begin{equation}\label{eq:d wt phi}d\wt\psi(\wt x)X_{\wt H_t}(\wt x)=\big(X_{H_t}(x),0\big).\end{equation}
By (\ref{eq:wt psi wt U}) the map $\wt\psi:\wt U\to U\x W$ is a diffeomorphism. Hence (\ref{eq:d wt phi}) implies that
\begin{align*}\wt\psi_*X_{\wt H_t}&:=\wt\psi_*X_{\wt H_t|\wt U}\\
&=X_{H_t|U}\x0:\textrm{ vector field on }\wt\psi(\wt U)\\
&=U\x W.
\end{align*}
It follows that 
\[\textrm{domain of flow of }\wt\psi_*X_{\wt H}=\big(\wt\psi_*X_{\wt H_t}\big)_t=\left(\textrm{domain of flow of }\left(X_{H_t|U}\right)_t\right)\x W,\]
and that
\[\phi_{\wt\psi_*X_{\wt H}}^t=\phi_{H|U}^t\x\id\]
on this domain. Since by assumption $\phi^t_H|_N\in\UU$ and $y_0\in K$, this point lies in the domain of the flow of $\left(X_{H_t}|U\right)_t$. Our assumption $\pi\circ\psi(K)\sub W$ implies that $\pi\circ\psi(y_0)\in W$. It follows that
\[\wt\psi\big(\wt x_0=(y_0,y_0)\big)=\big(y_0,\pi\circ\psi(y_0)\big)\]
lies in the domain of the flow of $\wt\psi_*X_{\wt H}$. This implies that $\wt x_0$ lies in the domain of the flow of $\big(X_{\wt H_t|\wt U}\big)_t$. The statement of Claim \ref{claim:phi wt H} follows.
\end{proof}
Since $\wt N$ is compact, using Claim \ref{claim:phi wt H}, there exists a compact neighbourhood $\wt K$ of $\wt N$ in $\wt M$ that is contained in the domain of $\phi_{\wt H}^1$. By Lemma \ref{le:compact support} below there exists a function $\hhat H\in C^\infty\big([0,1]\x\wt M,\R\big)$ with compact support, such that
\[\phi_{\hhat H}^t=\phi_{\wt H}^t\textrm{ on }\wt K,\quad\forall t\in[0,1].\]
We denote by $0_{\wt N}$ the zero-section of $T^*\wt N$. By our assumption there exists an open neighbourhood $U'\sub T^*\wt N$ of $0_{\wt N}$ and a symplectomorphism $\chi:U'\to\wt M$ that equals the identity on $\wt N$. (Recall that using Weinstein's neighbourhood theorem, we have shrunk $\wt M$ in such a way that such a $\chi$ exists.) We define the function $H'_t:T^*\wt N\to\R$ by
\[H'_t:=\left\{\begin{array}{ll}
\hhat H_t\circ\chi&\textrm{on }U',\\
0&\textrm{outside }U'.
\end{array}\right.\]
We have
\[\chi\big(0_{\wt N}\cap(\phi_{H'}^1)^{-1}(0_{\wt N})\big)=\wt N\cap(\phi_{\hhat H}^1)^{-1}(\wt N)=\wt N\cap(\phi_{\wt H}^1)^{-1}(\wt N).\]
By Lemma \ref{le:wt x} we have
\[\pr_1\big(\wt N\cap(\phi_{\wt H}^1)^{-1}(\wt N)\big)\sub\Fix(\phi_H^1,N).\]
Since $\chi$ and $\pr_1:\wt N\to M$ are injective, it follows that
\begin{equation}\label{eq:Fix}\big|\Fix(\phi_H^1,N)\big|\geq\big|0_{\wt N}\cap(\phi_{H'}^1)^{-1}(0_{\wt N})\big|.\end{equation}
It follows from \cite[Theorem 2]{HoLag} 
that 
\[\big|0_{\wt N}\cap(\phi_{H'}^1)^{-1}(0_{\wt N})\big|\geq\cl(\wt N)=\cl(N),\]
and therefore
\[\big|\Fix(\phi_H^1,N)\big|\geq\cl(N),\]
i.e., \textbf{the inequality (\ref{eq:cl N}) holds}.\\

Assume now that $(\phi^1_H,N)$ is $\OO$-nondegenerate. \textbf{We show that (\ref{eq:Fix phi 1 N}) holds.} Let
\[\wt x\in \wt N\cap(\phi_{\wt H}^1)^{-1}(\wt N).\]
We check the hypotheses of the implication $\then$ of Lemma \ref{le:N phi} with
\begin{eqnarray*}&U:=M,\quad\wt U:=\dom\big(\phi_{\wt H}^1\big),\quad\phi_t:=\phi_H^t,\quad\wt\phi_t:=\phi_{\wt H}^t,&\\
&y:[0,1]\to N,\quad y(t):=\pr_2\circ\phi_{\wt H}^t(\wt x).&
\end{eqnarray*}
The conditions (\ref{eq:pr 1 wt U sub U},\ref{eq:pr 1 wt phi t}) in this lemma follows from equality (\ref{eq:d pr 1}) in Lemma \ref{le:d pr 1}. The condition (\ref{eq:pr 2 wt phi t}) follows from statement (\ref{eq:d pr 2}) in Lemma \ref{le:d pr 1}. We denote by $i\in\{1,\ldots,\ell\}$ the smallest index, such that $y(0)=\pr_1(\wt x)\in\Int K_i$. Let $t\in[0,1]$. By Claim \ref{claim:phi wt H} we have
\[\phi_{\wt H}^t(\wt x)\in\wt U_i.\]
By (\ref{eq:wt psi wt U},\ref{eq:wt phi}) we have $\wt U_i\sub\textrm{domain of }\wt\psi_i\sub M\x V_i$. 
It follows that
\begin{equation}\label{eq:y V y i}y(t)\in\pr_2\big(\wt U_i\big)\sub V_i.\end{equation}
Since $i\in\{1,\ldots,\ell\}$ is the smallest index, such that $y(0)\in\Int K_i$, it follows that
\[\big(y(0),y(t)\big)\in\left(\Int K_i\wo\bigcup_{j=1,\ldots,i-1}\Int K_j\right)\x V_i\sub\OO.\]
Hence by our assumption that $(\phi^1_H,N)$ is $\OO$-nondegenerate, the triple $\big(\phi^1_H,N,y\big)$ is nondegenerate. Hence the hypotheses of the implication $\then$ of Lemma \ref{le:N phi} are satisfied. Applying this lemma, it follows that
\[T_{\wt x}\wt N\cap d\phi_{\wt H}^1(\wt x)^{-1}\big(T_{\phi_{\wt H}^1(\wt x)}\wt N\big)=\{0\}.\]
It follows that 
\[\wt N\pitchfork(\phi_{\wt H}^1)^{-1}(\wt N).\]
Since $\phi_{\hhat H}^t=\phi_{\wt H}^t$ on $\wt K$, the set $\wt K$ intersects the pre-images of $\wt N$ under $\phi_{\hhat H}^1$ and $\phi_{\wt H}^1$ in the same set. Since $\wt K$ is a neighbourhood of $\wt N$, it follows that 
\[\wt N\pitchfork(\phi_{\hhat H}^1)^{-1}(\wt N).\]
Since $0_{\wt N}=\chi^{-1}(\wt N)$ and $(\phi_{H'}^1)^{-1}(0_{\wt N})=\chi^{-1}\big((\phi_{\hhat H}^1)^{-1}(\wt N)\big)$, it follows that
\[0_{\wt N}\pitchfork(\phi_{H'}^1)^{-1}(0_{\wt N}).\]
Therefore, \cite[Theorem 1]{FlLag} implies that
\[\big|0_{\wt N}\cap(\phi_{H'}^1)^{-1}(0_{\wt N})\big|\geq\sum_{i=0}^{\dim 0_{\wt N}}b_i\left(0_{\wt N},\Z_2\right)=\sum_{i=0}^{\dim N}b_i(N,\Z_2).\]
Combining this with (\ref{eq:Fix}), the claimed inequality (\ref{eq:Fix phi 1 N}) follows. It follows that $\UU$ and $\OO$ satisfy condition (\ref{thm:leafwise:Fix}) of Theorem \ref{thm:leafwise}.\\

\textbf{We show that $\UU$ and $\OO$ satisfy condition (\ref{thm:leafwise:dense}) of Theorem \ref{thm:leafwise}.} We choose a compact neighbourhood $K'\sub M$ of $K:=N$, and $\ell$ and $H$ as in Lemma \ref{le:H}. Since $H_0\const0$, we have $\phi_{\wt H_0}^1=\id:\wt M\to\wt M$. Therefore, using compactness of $\wt N$, by restricting the function $H$ to the product of some ball in $\R^\ell$ and $M$, and rescaling, we may assume \Wlog that $\wt N\sub\dom(\phi_{\wt H_a}^1)$, for every $a\in B^\ell_1$. We define
\begin{eqnarray}\nn&A:=B^\ell_1,\quad X:=N,\quad Z:=\wt M,&\\
\label{eq:f}&f:A\x X\to Z,\,f(a,x):=\phi_{\wt H_a}^1(x,x),\quad Y:=\wt N.&
\end{eqnarray}
\begin{claim}\label{claim:df} For every 
$x\in N$ the map
\[df(0,x):\R^\ell\x T_xN\to T_{(x,x)}\wt M\]
is surjective.
%
%
\end{claim}

\begin{proof}[Proof of Claim \ref{claim:df}] We write $\wt x:=(x,x)$. Let $\wt v=(v,w)\in T_{\wt x}\wt M$. Since $(w,w)\in T_{\wt x}\wt N\sub T_{\wt x}\wt M$, we have
\begin{equation}\label{eq:v w}(v-w,0)=\wt v-(w,w)\in T_{\wt x}\wt M.\end{equation}
It follows from equality (\ref{eq:d pr 1}) in Lemma \ref{le:d pr 1} that $\pr_1\circ f(a,y)=\phi_{H_a}^1(y)$, for every $a\in B^\ell_1$, $y\in N$. Hence it follows from (\ref{eq:surj}) that there exists $\hhat a\in T_0B^\ell_1=\R^\ell$, such that
\begin{equation}\label{eq:d pr 1 d a f 0 x}d_a(\pr_1\circ f)(0,x)\hhat a=v-w.\end{equation}
(Here $d_a$ denotes the derivative \wrt $a$.) We denote
\[v_i:=d\pr_id_af(0,x)\hhat a=d_a(\pr_i\circ f)(0,x)\hhat a,\quad\forall i=1,2.\]
Since $d_af(0,x)\hhat a\in T_{\wt x}\wt M$, it follows from (\ref{eq:d pr 1 d a f 0 x},\ref{eq:v w}) that
\begin{equation}\label{eq:0}(0,v_2)=d_af(0,x)\hhat a-\big(v_1=v-w,0\big)\in T_{\wt x}\wt M.\end{equation}
Since $a\mapsto H_a$ is the restriction of a linear map to $B^\ell_1$, the same holds for the map $a\mapsto\wt H_a$. It follows that
\[\phi_{\wt H_{t\hhat a}}^1(x)=\phi_{t\wt H_{\hhat a}}^1(x)=\phi_{\wt H_{\hhat a}}^t(x)\]
(where defined), hence
\[\left.\frac d{dt}\right|_{t=0}\phi_{\wt H_{t\hhat a}}^1(x)=X_{\wt H_{\hhat a}}(x),\]
and therefore,
\[v_2=d\pr_2\left.\frac d{dt}\right|_{t=0}\phi_{\wt H_{t\hhat a}}^1(x)=d\pr_2X_{\wt H_{\hhat a}}(x).\]
Therefore statement (\ref{eq:d pr 2}) in Lemma \ref{le:d pr 1} implies that
\begin{equation}\label{eq:d pr 2 d a f 0 x}v_2\in T_xN^\om.\end{equation}
Since $\wt M$ is symplectic, we have 
$T_{\wt x}\wt M\cap\big(\{0\}\x T_xN^\om\big)=\{0\}$. 
Combining this with (\ref{eq:0},\ref{eq:d pr 2 d a f 0 x}), it follows that
\begin{equation}\label{eq:d pr 2 0}v_2=0.\end{equation}
Since $f(0,y)=(y,y)$, for every $y\in N$, we have $d_xf(0,x)w=(w,w)$. Combining this with (\ref{eq:d pr 1 d a f 0 x},\ref{eq:d pr 2 0}), it follows that
\[df(0,x)(\hhat a,w)=d_af(0,x)\hhat a+d_xf(0,x)w=(v-w,0)+(w,w)=\wt v.\]
Hence $df(0,x)$ is surjective. This proves Claim \ref{claim:df}.
\end{proof}

Using Claim \ref{claim:df}, by restricting the function $H$ to the product of some ball in $\R^\ell$ and $M$, and rescaling, we may assume \Wlog $df(a,x)$ is surjective for every $(a,x)\in B^\ell_1\x N$. 

Recall the definition (\ref{eq:Ham M om U}) of $\Ham(M,\om,\UU)$. Let $G\in C^\infty\big([0,1]\x M,\R\big)$ be a function whose Hamiltonian flow exists globally, such that $\phi_G^t|_N\in\UU$, for every $t\in[0,1]$. Let $\VV$ be a neighbourhood of
\[\phi_0:=\phi_G^1\]
in $\Ham(M,\om,\UU)$ in the strong $C^\infty$-topology.
\begin{claim}\label{claim:phi} There exists $\phi\in\VV$, such that $\phi$ is $\OO$-nondegenerate.
\end{claim}

\begin{proof}[Proof of Claim \ref{claim:phi}] We choose a neighbourhood $\VV'$ of $\phi_G^1$ in $C^\infty(M,M)$ in the strong $C^\infty$-topology, such that $\VV=\VV'\cap\Ham(M,\om,\UU)$. We define
\begin{equation}\label{eq:A 0}A_0:=\left\{a\in A=B^\ell_1\,\big|\,\phi_0\circ\phi_{H_a}^1\in\VV',\,\phi_{G\#H_a}^t\big|_N\in\UU,\,\forall t\in[0,1]\right\}.\end{equation}
Here 
\[(G\#H_a)_t:=\left\{\begin{array}{ll}
2H_a,&\textrm{for }t\in\left[0,\frac12\right],\\
2G_{2t-1},&\textrm{for }t\in\left(\frac12,1\right],
\end{array}\right.\]
%
%
Let $a\in A$. Since $H_a$ has support in the compact set $K'$, its Hamiltonian flow exists globally and is surjective for all times. It follows that the same is true for $G\#H_a$. Hence the restriction $\phi_{G\#H_a}^t\big|_N$ in the definition of $A_0$ makes sense. Furthermore, $\phi_{G\#H_a}^1=\phi_0\circ\phi_{H_a}^1\in\Ham(M,\om)$, and therefore
\begin{equation}\label{eq:A 0 VV}A_0\sub\left\{a\in A\,\big|\,\phi_0\circ\phi_{H_a}^1\in\VV
\right\}.\end{equation}
Since $\phi_{H_0\const0}^1=\id$, the set of all $a\in A$, such that $\phi_0\circ\phi_{H_a}^1\in\VV'$, is a neighbourhood of $0$. (Here we use that the solution of an ordinary differential equation depends smoothly on the initial data and given parameters, if the coefficients of the equation depend smoothly on the point in space and on the parameters.) Similarly, the set of all $a\in A$, such that $\phi_{G\#H_a}^t|_N\in\UU$, for every $t$, is a neighbourhood of $0$. Using (\ref{eq:A 0}), it follows that $A_0$ is a neighbourhood of $0$. 

Since $df(a,x)$ is surjective for every $(a,x)\in A\x(X=N)$, by Proposition \ref{prop:transv}, the set
\[A_1:=\left\{a\in A\,\big|\,f(a,\cdot)\pitchfork\left(\phi_{\wt G}^1\right)^{-1}(\wt N)\right\}\]
%
%
is residual. Here $\wt G$ is defined as in (\ref{eq:wt H}). By Baire's category theorem $A_1$ is therefore dense in $A$. It follows that there exists $a\in A_0\cap A_1$. We define
\[\phi:=\phi_0\circ\phi_{H_a}^1.\]
It follows from (\ref{eq:A 0 VV}) that $\phi\in\VV$. Claim \ref{claim:phi} now follows from the next claim.

\begin{claim}\label{claim:phi V}$\phi$ is $\OO$-nondegenerate. 
\end{claim}

\begin{pf}[Proof of Claim \ref{claim:phi V}] Let $F\sub N$ be an isotropic leaf and
\begin{equation}\label{eq:x F}x\in C^\infty([0,1],F)\end{equation}
a path, such that (\ref{eq:x 0 U phi x},\ref{eq:x 0 x t O}) hold with $U=\dom(\phi)$. We show that the triple $\big(\phi,N,x\big)$ is nondegenerate. We define
\begin{equation}\label{eq:y}y:[0,1]\to F,\quad y(t):=\pr_2\circ\phi_{\wt{G\#H_a}}^t\big(x(0),x(0)\big).\end{equation}
We show that this map is well-defined. Since $a\in A_0$, we have that $\phi_{G\#H_a}^t|_N\in\UU$, for every $t$. Hence by Claim \ref{claim:phi wt H} applied with $H$ replaced by $G\#H_a$, we have $\big(x(0),x(0)\big)\in\wt N\sub\dom\big(\phi_{\wt{G\# H_a}}^1\big)$. Hence the right hand side in (\ref{eq:y}) makes sense. It follows from statement (\ref{eq:d pr 2}) in Lemma \ref{le:d pr 1} that $y(t)\in F$, for every $t\in[0,1]$. Hence $y$ is well-defined. 
\begin{claim}\label{claim:phi N y} The triple $(\phi,N,y)$ is nondegenerate.
\end{claim}
\begin{proof}[Proof of Claim \ref{claim:phi N y}] We check the hypotheses of the implication ``$\follows$'' of Lemma \ref{le:N phi} with
\begin{eqnarray*}&U:=M,\quad\wt U:=\dom\left(\phi_{G\#H_a}^1\right),\quad\phi_t:=\phi_{G\#H_a}^t,\quad\wt\phi_t:=\phi_{\wt{G\#H_a}}^t,&\\
&\wt x:=\big(x(0),x(0)\big).&
\end{eqnarray*}
It follows from equality (\ref{eq:d pr 1}) in Lemma \ref{le:d pr 1} that the hypotheses (\ref{eq:pr 1 wt U sub U},\ref{eq:pr 1 wt phi t}) of Lemma \ref{le:N phi} are satisfied. By statement (\ref{eq:d pr 2}) in Lemma \ref{le:d pr 1} the hypothesis (\ref{eq:pr 2 wt phi t}) is satisfied. Since $a\in A_1$, we have
\begin{equation}\label{eq:f pitchfork}f(a,\cdot)\pitchfork\left(\phi_{\wt G}^1\right)^{-1}(\wt N).
\end{equation}
By (\ref{eq:f}) we have
\[f(a,\cdot)\circ\pr_1=\phi_{\wt H_a}^1:\wt N\to\wt M.\]
Using (\ref{eq:f pitchfork}) and that $\pr_1:\wt N\to N$ is a diffeomorphism, it follows that
\begin{eqnarray*}&\left.\phi_{\wt H_a}^1\right|_{\wt N}\pitchfork\left(\phi_{\wt G}^1\right)^{-1}(\wt N),\quad\textrm{i.e.,}&\\
&\wt N\pitchfork\left(\left(\phi_{\wt H_a}^1\right)^{-1}\left(\left(\phi_{\wt G}^1\right)^{-1}(\wt N)\right)=\left(\phi_{\wt G}^1\circ\phi_{\wt H_a}^1\right)^{-1}(\wt N)\right).&
\end{eqnarray*}
(The composed function on the right hand side is defined on the pre-image of $\dom\left(\phi_{\wt G}^1\right)$ under $\phi_{\wt H_a}^1$.) Since 
\begin{align*}\phi_{\wt G}^1\circ\phi_{\wt H_a}^1&=\phi_{\wt G\#\wt H_a}^1\\
&=\phi_{\wt{G\#H_a}}^1\\
&=\wt\phi_1,
\end{align*}
it follows that
\[\wt N\pitchfork\wt\phi_1^{-1}(\wt N).\]
Hence the hypotheses of the implication ``$\follows$'' of Lemma \ref{le:N phi} are satisfied. Applying that lemma, it follows that $\big(\phi=\phi_1,N,y\big)$ is nondegenerate. This proves Claim \ref{claim:phi N y}.
\end{proof}

Recall that the points $y_1,\ldots,y_\ell\in N$ are chosen, such that (\ref{eq:N sub bigcup}) holds, and that $x$ is a path as in (\ref{eq:x F}). We denote by $i$ the smallest index, such that $x(0)\in\Int K_i$. 
\begin{claim}\label{claim:x y} We have
\[x([0,1]),\,y([0,1])\sub V_i.\]
\end{claim}
(Recall that $V_i=V_{y_i}$ was chosen as in Claim \ref{claim:V y}.)
\begin{proof}[Proof of Claim \ref{claim:x y}] Let $t\in[0,1]$. By our assumption (\ref{eq:x 0 x t O}), we have $\big(x(0),x(t)\big)\in\OO$. Using the definition (\ref{eq:OO}) of $\OO$ and our choice of $i$, it follows that $x(t)\sub V_i$. This proves that $x([0,1])\sub V_i$. 

By (\ref{eq:y}) we have $y(0)=x(0)\in\Int K_i$. Since $a\in A_0$, we have that $\phi_{G\#H_a}^t|_N\in\UU$, for every $t$. Hence by Claim \ref{claim:phi wt H}, we have
\[\phi_{\wt{G\#H_a}}^t\big(x(0),x(0)\big)\in\wt U_i,\quad\forall t\in[0,1].\]
By (\ref{eq:wt psi wt U},\ref{eq:wt phi}) we have $\wt U_i\sub\textrm{domain of }\wt\psi_i\sub M\x V_i$. 
Using (\ref{eq:y}), it follows that
\[y([0,1])\sub\pr_2\big(\wt U_i\big)\sub V_i.\]
This proves Claim \ref{claim:x y}.
\end{proof}
By (\ref{eq:y}) we have
\[x(0)=y(0).\]
Using (\ref{eq:x 0 U phi x}), we have
\[\big(x(1),y(1)\big)=\big(\phi_{G\#H_a}^1\circ x(0),y(1)\big)=\phi_{\wt{G\#H_a}}^1\big(x(0),x(0)\big)\in\wt M.\]
Combining this with Claim \ref{claim:x y}, the fact that $x$ and $y$ are paths in the same leaf $F$,  and (\ref{eq:z z'}), it follows that
\[x(1)=y(1).\]
Combining this with Claim \ref{claim:x y}, Remark \ref{rmk:hol}, and surjectivity of the foliation chart $\psi_i:V_i\to\R^{2m}\x\R^k$, it follows that
\[\hol^{N,\om}_x=\hol^{N,\om}_y.\]
(This statement makes sense, since $x(i)=y(i)$, for $i=0,1$.) Since $\big(\phi,N,y\big)$ is nondegenerate, using Claim \ref{claim:phi N y}, it follows that $\big(\phi,N,x\big)$ is nondegenerate, as desired. It follows that $\phi$ is $\OO$-nondegenerate. This proves Claim \ref{claim:phi V} and completes the proof of Claim \ref{claim:phi}.
\end{pf}
\end{proof}

Since $\phi\in\VV$, it follows from Claim \ref{claim:phi} that the set
\[\big\{\phi\in\Ham(M,\om,\UU)\,\big|\,\phi\textrm{ is }\OO\textrm{-nondegenerate}\big\}\]
is dense in $\Ham(M,\om,\UU)$ in the strong $C^\infty$-topology. This proves (\ref{thm:leafwise:dense}) and completes the proof of Theorem \ref{thm:leafwise}.
\end{proof}
\begin{rmk}[method of proof of Theorem \ref{thm:leafwise}]\label{rmk:regular} The method of proof of Theorem \ref{thm:leafwise}(\ref{thm:leafwise:Fix}) refines the technique used in the proof of \cite[Theorem 1.1]{ZiLeafwise} in the following sense. Assume that $N$ is regular (i.e., ``fibering'') in the sense that there exists a manifold structure on the set $N_\om$ of isotropic leaves of $N$, such that the canonical projection $\pi_N:N\to N_\om$ is a smooth submersion. We denote by $\om_N$ the unique symplectic form on $N_\om$ that pulls back to $\iota_N^*\om$ under $\pi_N$. 
We equip the product $\hhat M:=M\x N_\om$ with the symplectic form $\hhat\om:=\om\oplus(-\om_N)$. 

In \cite{ZiLeafwise} the symplectic manifold $(\hhat M,\hhat\om)$ was used to prove a lower bound on $\big|\Fix(\phi,N)\big|$ for a regular coisotropic submanifold $N$. On the other hand, the proof of Theorem \ref{thm:leafwise} is based on the construction of a certain symplectic submanifold $\wt M\sub M\x N$ (see Lemma \ref{le:wt M}), which can be viewed as a local version of $(\hhat M,\hhat\om)$. More precisely, if $N$ is regular then $\wt M$ can be symplectically embedded into $\hhat M$ via the map
\[(x,y)\mapsto(x,N_y),\]
where $N_y$ denotes the isotropic leaf of $N$ through $y$. 
\end{rmk}
\begin{Rmk}[simplifying the proof] Suppose that we only want to show that there exist $\UU$ and $\OO$ satisfying condition (\ref{thm:leafwise:Fix}) in Theorem \ref{thm:leafwise} (but not necessarily condition (\ref{thm:leafwise:dense})). For this we do not need Claim \ref{claim:V y}. Instead we may choose an arbitrary open neighbourhood $V_y$ of $y$ in $N$ and an arbitrary foliation chart $\psi_y:V_y\to\R^{2m}\x\R^k$.

We may simplify the proof further by choosing $\OO:=N\x N$. This means that we need to show the inequality (\ref{eq:Fix phi 1 N}) only if $(\phi^1,N)$ is $N\x N$-nondegenerate, i.e., nondegenerate in the sense of \cite{ZiLeafwise}.
\end{Rmk}
\appendix
\section{Auxiliary results}
In the proof of Lemma \ref{le:wt M} we used the following.
\begin{lemma}[local embedding]\label{le:emb}Let $M$ and $M'$ be manifolds (without boundary), $K\sub M$ a compact subset, and $f:M\to M'$ a smooth map whose restriction to $K$ is injective, such that $df(x)$ is injective for every $x\in K$. Then there exists an open neighbourhood $U\sub M$ of $K$, such that $f|_U$ is a smooth embedding.
\end{lemma}
In the proof of this lemma we will use the following notation. Let $(X,d)$ be a metric space and $A,B\sub X$. We denote
\[d(A,B):=\inf_{(a,b)\in A\x B}d(a,b).\]
\begin{proof}[Proof of Lemma \ref{le:emb}]\setcounter{claim}{0} We show that there exists an open neighbourhood $U_0$ of $K$ on which $f$ is injective. Since $f$ is continuous and $f|_K$ is injective, the set
\[S:=\big\{(x,y)\in M\x M\,\big|\,x=y\textrm{ or }f(x)\neq f(y)\big\}\]
is a (possibly nonopen) neighbourhood of
\[\big\{(x,y)\in K\x K\,\big|\,x\neq y\big\}\]
in $M\x M$. By the Immersion Theorem every point in $K$ admits an open neighbourhood in $M$ on which $f$ is injective. It follows that $S$ is a neighbourhood of 
\[\big\{(x,x)\,\big|\,x\in K\big\},\]
and therefore of $K\x K$.
\begin{claim}\label{claim:U 0}There exists an open neighbourhood $U_0\sub M$ of $K$ such that $U_0\x U_0\sub S$.
\end{claim}
\begin{proof}[Proof of Claim \ref{claim:U 0}]We choose a distance function $d$ on $M$ that induces the topology. We define the distance function $\wt d$ on $\wt M:=M\x M$ by
\[\wt d\big((x,y),(x',y')\big):=d(x,x')+d(y,y').\]
Since $\wt K:=K\x K$ is compact, there exists a constant $\eps>0$, such that the closed $\eps$-neighbourhood of $\wt K$,
\[\wt K_\eps:=\left\{\wt a\in\wt M\,\Big|\,\exists\wt b\in\wt K:\,\wt d(\wt a,\wt b)\leq\eps\right\}\]
is compact. The same then holds for $\wt K_\eps\wo\Int S$. Hence $\wt d$ attains its minimum on $\wt K\x\big(\wt K_\eps\wo\Int S\big)$. This minimum is positive, since if $(\wt a,\wt b)\in\wt K\x\wt M$ is such that $\wt d(\wt a,\wt b)=0$ then $\wt b=\wt a\in\wt K\sub\Int S$. It follows that 
\[\eps_0:=\wt d\left(\wt K,\wt M\wo\Int S\right)>0.\]
We define
\[U_0:=\left\{x\in M\,\Big|\,\exists y\in K:\,d(x,y)<\frac{\eps_0}2\right\}.\]
This is an open neighbourhood of $K$. Let $(x,y)\in U_0\x U_0$. Then 
\[\wt d\big(\{(x,y)\},\wt K\big)<2\cdot\frac{\eps_0}2,\]
hence $(x,y)\in\Int S$. Hence the set $U_0$ has the desired properties. This proves Claim \ref{claim:U 0}.
\end{proof}
We choose $U_0$ as in this claim. The restriction $f|_{U_0}$ is injective. 
\begin{claim}\label{claim:U 1} The set
\[U_1:=\big\{x\in M\,\big|\,df(x)\textrm{ injective}\big\}\]
is an open neighbourhood of $K$.
\end{claim}
\begin{proof}[Proof of Claim \ref{claim:U 1}] This set contains $K$. To show openness, consider first the \textbf{case $M=\R^n$, $M'=\R^{n'}$}. The set 
\[\big\{\textrm{injective linear map from }\R^n\textrm{ to }\R^{n'}\big\}\]
is open in the space of all linear maps from $\R^n$ to $\R^{n'}$. Since
\[df:\R^n\to\big\{\textrm{linear map:}\R^n\to\R^{n'}\big\}\]
is continuous, it follows that $U_1$ is open. This proves the statement in the case $M=\R^n$, $M'=\R^{n'}$. The general situation can be reduced to this case by using charts. This proves Claim \ref{claim:U 1}.
\end{proof}
We choose $U_1$ as in this claim and an open neighbourhood $U\sub M$ of $K$ whose closure is compact and contained in $U_0\cap U_1$. The restriction of $f$ to $U$ is proper onto its image. It follows that this restriction is a smooth embedding. This proves Lemma \ref{le:emb}.
\end{proof}
The next result was also used in the proof of Lemma \ref{le:wt M}. Let $V$ be a finite dimensional real vector space and $b:V\x V\to\R$ a bilinear map. Recall that $b$ is called \emph{nondegenerate} iff the map
\[V\ni v\mapsto b(v,\cdot)\in V^*\] 
is an isomorphism.
\begin{lemma}[nondegeneracy]\label{le:nondeg} Let $M$ be a manifold and $b$ a field of bilinear forms on $M$. 
Then the set
\[S:=\big\{x\in M\,\big|\,b_x\textrm{ is nondegenerate}\big\}\]
is open.
\end{lemma}
\begin{proof}[Proof of Lemma \ref{le:nondeg}] Consider the \textbf{case} $M=\R^n$. Then
\[S=\Big\{x\in\R^n\,\big|\,b_x\in\big\{\textrm{nondegenerate bilinear map: }\R^n\x\R^n\to\R\big\}\Big\}.\]
This set is open in $\R^n$, since the set
\[\big\{\textrm{nondegenerate bilinear map: }\R^n\x\R^n\to\R\big\}\]
is open in the set of all bilinear maps and the map $x\mapsto b_x$ is continuous. This proves the statement in the case $M=\R^n$. The general situation can be reduced to this case by using charts for $M$. This proves Lemma \ref{le:nondeg}.
\end{proof}
In the proof of Theorem \ref{thm:leafwise} (Claim \ref{claim:OO}) we used the following.
\begin{rmk}\label{rmk:X Y} Let $I\sub J\sub\N$ be finite subsets, for $i\in J$ let $A_i\sub B_i$ be sets, and let $i_0\in I$. Then we have 
\[X:=A_{i_0}\wo\bigcup_{j\in J\wo I}B_j\sub Y:=\bigcup_{i\in I}\left(A_i\wo\bigcup_{j\in J:\,j<i}A_j\right).\]
To see this, let $x\in X$. We define
\[i_1:=\min\big\{i\in I\,\big|\,x\in A_i\big\}.\]
For every $j\in J$ satisfying $j<i_1$, we have $x\not\in A_j$. (This follows by considering the cases $j\in I$ and $j\not\in I$ separately.) It follows that 
\[x\in A_{i_1}\wo\bigcup_{j\in J:\,j<i_1}A_j\sub Y.\]
This proves that $X\sub Y$.
\end{rmk}
In the proof of Theorem \ref{thm:leafwise} we also used the following.
\begin{lemma}[Hamiltonian flow]\label{le:compact support}Let $(M,\om)$ be a symplectic manifold (without boundary), $H\in C^\infty\big([0,1]\x M,\R\big)$, and $K\sub\bigcap_{t\in[0,1]}\dom(\phi_H^t)=\dom(\phi_H^1)$ a compact subset. Then there exists a function
\[\hhat H\in C^\infty\big([0,1]\x M,\R\big)\]
with compact support, such that
\[\phi_{\hhat H}^t=\phi_H^t\textrm{ on }K,\quad\forall t\in[0,1].\]
\end{lemma}
\begin{proof}[Proof of Lemma \ref{le:compact support}] This follows from a cut-off argument as in the proof of \cite[Lemma 35]{SZ}. 
%
\end{proof}
\section*{Acknowledgments}
I would like to thank L.~Buhovsky for making me aware that the answer to the question in Remark \ref{rmk:C0} is yes for coordinate space and for closed surfaces. Some correspondence with S.~Seyfaddini, L.~Polterovich and S.~M\"uller regarding this question is also gratefully acknowledged. Finally, I thank G.~Cavalcanti for stimulating discussions and one referee for his/ her detailed and useful comments.

\end{document}